\documentclass[12pt]{amsart}
\usepackage[margin=1.2in]{geometry}
\usepackage{color}
\title [Reconstruction algorithms for source term recovery] {Reconstruction algorithms for source term recovery from dynamical samples in catalyst models}
\author[Akram Aldroubi, Le Gong, Ilya Krishtal, Brendan Miller, Sumati Thareja]{A. Aldroubi, L. Gong, I. Krishtal, B. Miller, S. Thareja}
\usepackage{hyperref}

\usepackage{amsthm, amsfonts, mathrsfs} 
\usepackage{ dsfont } 
\usepackage{amssymb}
\usepackage{enumerate}
\usepackage{graphicx}
\usepackage{subfig}

\usepackage{float}
\usepackage{bbm}
\usepackage{amsmath}
\usepackage{comment}
\usepackage{hyperref}
\usepackage{listings}

\usepackage{color}

\usepackage{diagbox}
\usepackage[dvipsnames]{xcolor}
\usepackage{algorithm}
\usepackage{booktabs}
\usepackage{algorithmicx}
\usepackage[noend]{algpseudocode}
\usepackage{tikz}
\allowdisplaybreaks

\hypersetup{
    colorlinks,
    citecolor=blue,
    filecolor=blue,
    linkcolor=blue,
    urlcolor=blue Ce
}

\newtheorem{theorem}{Theorem}[section]
\newtheorem{proposition}[theorem]{Proposition}
\newtheorem{lemma}[theorem]{Lemma}

\newtheorem{remark}[theorem]{Remark}

\newcommand{\HH}{\mathcal{H}}

\newcommand{\R}{\mathbb{R}}
\newcommand{\Z}{\mathbb{Z}}

\newcommand{\G}{\mathcal{G}}
\newcommand{\T}{\mathcal{T}}

\newcommand{\estime}{\mathfrak t}
\newcommand{\esshape}{\mathfrak{f}_j(g)}
\newcommand{\ang}[1]{\langle #1 \rangle}

\usepackage[foot]{amsaddr}
\newcounter{lst}

\definecolor{aacolor}{rgb}{0.05, 0.75, 1}

\definecolor{ikcolor}{rgb}{1, 0., 0.}

\definecolor{bmcolor}{rgb}{0.9, 0.3, 0}

\begin{document}
	\date{}



%
%
\address{\textrm{(Akram Aldroubi)}
Department of Mathematics,
Vanderbilt University,
Nashville, TN, 37240, USA}
\email{akram.aldroubi@vanderbilt.edu}

\address{\textrm{(Le Gong)}
	Department of Mathematics,
	Vanderbilt University,
	Nashville, TN, 37240, USA}
\email{le.gong@vanderbilt.edu}

\address{\textrm{(Ilya Krishtal)}
        Department of Mathematical Sciences, Northern Illinois University,
        Dekalb, IL, 60115, USA}
\email{ikrishtal@niu.edu}

\address{\textrm{(Brendan Miller)}
        Department of Mathematical Sciences, Northern Illinois University,
        Dekalb, IL, 60115, USA}
\email{bmiller14@niu.edu}

\address{\textrm{(Sumati Thareja)}
    Department of Mathematics,
	Vanderbilt University,
	Nashville, TN, 37240, USA}
\email{sumati.thareja@vanderbilt.edu}


\maketitle

\begin{abstract}
{This paper investigates the problem of recovering source terms in abstract initial value problems (IVP) commonly used to model various scientific phenomena in physics, chemistry, economics, and other fields. We consider source terms of the form 
$F=h+\eta$, where $\eta$ is a Lipschitz continuous background source. The primary objective is to estimate the unknown parameters of non-instantaneous sources 
$h(t)=\sum\limits_{j=0}^M h_je^{-\rho_j(t-t_j)}\chi_{[t_j,\infty)}(t)$, such as the decay rates, initial intensities and activation times. 
We present two novel recovery algorithms that employ distinct sampling methods of the solution  of the IVP. Algorithm 1 combines discrete and weighted average measurements, whereas Algorithm 2 uses a different variant of weighted average measurements. We analyze the performance of these algorithms, providing upper bounds on the recovery errors of the model parameters. Our focus is on the structure of the dynamical samples used by the algorithms and on the error guarantees they yield.}
\end{abstract}

\section{Introduction}

Numerous scientific phenomena in physics, chemistry, economics, and other fields can be effectively modeled using abstract initial value problems (IVP) such as:
\begin{equation}\label{DFM}
	\begin{cases}
	\dot{u}(t)=Au(t)+F(t)\\
	u(0)=u_0,
	\end{cases}
	\quad t\in\mathbb R_+,\ u_0\in\HH.
\end{equation}  
The above problem typically describes propagation of a phenomenon in time, represented by the variable $t\in\R_+$. 
The function $u$ in \eqref{DFM} is vector-valued; in this paper, we shall assume that $u(t)$ is a vector in some Hilbert space $\HH$. Most often, $\HH$ will be a space of functions on a subset of $\R^d$, such as $L^2([0,1]^d)$. 
We then have that $\dot{u}: \R_+\to\HH$ is the time derivative of $u$ and $F:\mathbb{R}_+\rightarrow\HH$ is  the forcing term. The operator $A: D(A)\subseteq \HH\to\HH$ is assumed to be a generator of a strongly continuous semigroup $T:\mathbb R_+\to B(\HH)$, where $B(\HH)$ denotes the Banach algebra of all bounded linear operators on $\HH$. 

One common example of \eqref{DFM} is the diffusion equation, which is widely used to model, for instance, the dispersion of biochemical wastes or the spread of fungal diseases \cite{egan1972numerical, langendoen2006murphy, 1420811}. In these cases, the vector $u(t)$ describes the concentration of the released substance in various spatial locations at time $t$, thus giving the researchers a model for substance propagation across both space and time. Problem \eqref{DFM} may also describe an ongoing chemical reaction or a combined effect of various medications over time. In these cases, the vector $u(t)$ may be finite dimensional and describe the amounts of various chemicals in the system at time $t$. 

 To study scientific phenomena, researchers commonly utilize sensors to gather spatiotemporal measurements   within a specified area. Within the above framework, the gathered measurements are samples of the solution $u$ of \eqref{DFM} and the goal of study is to determine more information about $u$ itself or about other parameters of \eqref{DFM}. Thus  effective and robust sampling and reconstruction methods greatly  influence various practical applications and new approaches may lead to significant advances. 
 
 We place the above sampling-reconstruction problem within the framework of dynamical sampling. The framework encompasses a range of problems in which various samples of a signal  $u$ evolving over time under the influence of a linear operator $A$ 
 are used to recover various aspects of the system, such as $u_0$, $A$ or $F.$ Recovering the initial condition $u_0$ is known as the space-time trade-off sampling problem (see \cite{alddavkri13}). Recovering the linear operator $A$ gives rise to the system identification problem (see, e.g., \cite{AHKLLV18, aldkri16,  CT22, Tan17}). The problem investigated in this paper falls within the realm of identifying specific types of source terms $F$ that drive the dynamical system \cite{aldgonkri23, aldhuakor23}. Additionally, dynamical sampling problems are closely connected to multiple branches of mathematics, such as frame theory, control theory, functional analysis, and harmonic analysis (see \cite{AKh17, ACCP21, ashbrock2023dynamical, BH23, BK23, CMS21, cabrelli2022frames, CH19, CH23, DMM21, MMM21,  martin2023error,  FS19, GRUV15, KS19, Men22, MT23}).

In this paper, we will only consider sources of the form $F=h+\eta$, where $\eta$ is a Lipschitz continuous background source which is of no interest to us. In some cases, we are concerned with the classical solutions of \eqref{DFM} and will additionally assume that $\eta$ belongs to the Sobolev space $W^{1,1}_{loc}(\mathbb R_{+},\HH)$, which consists of all functions $f:\mathbb R_{+}\rightarrow\HH$ such that both $f$ and its weak first derivative are Lebesgue integrable on compact subsets of $\mathbb R_{+}$ \cite{evans2010partial}.

We will use various samples of $u$ collected from sensors to solve the problem of estimating certain parameters of $h.$
In the literature \cite{aldgonkri23, aldhuakor23, murdra1506}, two types of the sources $h$ are typically considered: instantaneous sources ($h(t)=\sum\limits_{j=0}^M h_j\delta(t-t_j)$) and non-instantaneous sources that exhibit exponential decay in intensity over time after activation ($h(t)=\sum\limits_{j=0}^M h_je^{-\rho_j(t-t_j)}\chi_{[t_j,\infty)}(t)$). The reconstruction of instantaneous sources has been thoroughly investigated in \cite{aldhuakor23}, where accurate estimates of both the initial intensities and activation times have been obtained. The scenario involving non-instantaneous sources with a known uniform decay rate (i.e., where $\rho_j=\rho$ for all $j$ is known) has been extensively explored in \cite{aldgonkri23}. In this paper, we will expand the study of non-instantaneous sources to estimate  all of their parameters.  Our focus here will be on non-instantaneous sources with varying unknown decay rates, which we would like to estimate together with their likewise unknown initial intensities and activation times.

Thus, the problem we study can be formalized as follows:
\begin{equation}\label{model}
\begin{cases}
\dot{u}(t)=Au(t)+\sum\limits_{j=0}^M h_je^{-\rho_j(t-t_j)}\chi_{[t_j, \infty)}(t)+\eta(t)\\
u(0)=u_0,\\
\end{cases}
\quad t\in\mathbb R_+,\ u_0\in \HH.
\end{equation} 
The parameters $\rho_j \in [\check{\rho}, \hat{\rho}]$, $t_j\in\mathbb R_{+}$  and $h_j\in\HH$, $j=0,1,\ldots, M$, are unknown and to be recovered (subject to known bounds $0 < \check{\rho} \le \hat{\rho}$). As mentioned above, $\eta: \mathbb R_{+} \to \HH$ is a Lipschitz continuous background source term, which is assumed to belong to $W^{1,1}_{loc}(\mathbb R_{+},\HH)$ whenever necessary.

As the goal of this paper includes recovering of the decay rates $\rho_j$, the approach employed in our previous work \cite{aldgonkri23} will not be applicable in this context.
In particular, the samples of $u$ we utilize here will be different from \cite{aldgonkri23}. In fact, we shall present two recovery algorithms that use different kinds of samples.
In Algorithm 1, we will employ a novel approach that combines discrete and  weighted average  measurements of $u$.
Algorithm 2 uses slightly different weighted average  measurements of $u$, this approach mimics that of \cite{aldhuakor23}.

 Algorithm 1 studies classical solutions and has the prerequisite that the initial condition $u_0$ must belong to the domain  $D(A)$, which is a dense set  within the Hilbert space $\HH.$ Algorithm 2  can be applied in a broader context of  mild solutions of \eqref{model}. 
Both algorithms improve estimates in \cite{aldgonkri23} for the case when $\check{\rho}=\hat{\rho}$.

We conclude the introduction with a few notes on paper organization and a small collection of useful facts from the theory of one-parameter operator semigroups.

\subsection{Paper Organization}

Section \ref{sec3} introduces two reconstruction algorithms designed to address Problem \eqref{model}, along with an exploration of the underlying model assumptions and key ideas. The performance analysis of the two algorithms is  presented in  Section \ref{sec4}. Specifically, Theorems \ref{thm1} and \ref{case1_theorem} provide upper bounds on the recovery errors of the model parameters for Algorithms 1 and 2, respectively. Proofs of the theorems appear in subsections of Section \ref{sec4}. In Section \ref{sec5}, the algorithms' performance is illustrated on a synthetic dynamical system. Finally, a few proofs of technical results appear in the Appendix.

\subsection{IVP toolkit}\label{stoolkit}

Here we remind the reader a few basic facts of operator semigroup theory and use them to provide a solution formula for IVPs of the form \eqref{model}. We refer to \cite{EN00} for more information.

A strongly continuous operator semigroup  is a map $T:\mathbb{R}_+\rightarrow B(\HH)$, 
which satisfies 
\begin {enumerate}[ (i)]
\item $T(0)=I$, 
\item $T(t+s)=T(t)T(s)$ for all $t,s\geq 0$, and
\item $\|T(t)x-x\|\rightarrow 0$ as $t\rightarrow 0$ for all $x\in \HH$.
\end {enumerate}

The operator $A$ is the (infinitesimal) generator of the semigroup $T$ if, given
\[D(A) = \left\{x\in \HH: \ \lim_{t\to 0^+} \frac1t(T(t)x-x) \mbox{ exists}\right\}, \]   $A$ satisfies 
 	\[Ah = \lim_{t\to 0^+} \frac1t(T(t)x-x),\ x\in D(A).\]
 For a strongly continuous semigroup, the domain $D(A)$ of its generator $A$ is known to be a dense subset of $\HH$. 

 According to \cite[p.~436]{EN00}  the (mild) solution of \eqref{DFM} can be represented  as
\begin{equation}\label{Solution}
	u(t):=T(t)u_0+\int_{0}^{t}T(t-s)F(s)ds.
\end{equation}
Therefore, the solution of \eqref{model} can be written as
\begin{equation}\label{solburst}
	u(t)=T(t)u_0+\sum\limits_{t_j<t}\int_{t_j}^{t}T(t-s)h_je^{-\rho_j(s-t_j)}ds+\int_0^{t}T(t-s)\eta(s)ds, \quad
	t\ge 0.
\end{equation}
Moreover, if we take $u_0\in D(A)$ and $\eta\in W^{1,1}_{loc}(\mathbb R_{+},\HH),$  then \eqref{solburst} is the unique classical solution of \eqref{model}, i.e.~$u\in C^1([0,\mathcal T],\HH)$ for any $\T > 0$.

\section{Reconstruction algorithms}\label{sec3}
In this section, we present two reconstruction algorithms for Problem \eqref{model} and some key ideas these algorithms are based upon. The algorithms' derivation and error guarantees will be presented in the following section.

Recall that  in \eqref{model} we represented the source term $F$ of \eqref{DFM} as a sum of two terms $F=h+\eta$, where $\eta$ is a background source term and $h$ is the term of interest, for which we would like to recover its key features. 

We assume that the background source $\eta$ is Lipschitz continuous:
\begin{equation} \label{LipBackgrnd}
  \|\eta(t+s)-\eta(t)\|\le Ls \quad \forall t,s\in \mathbb R_+.  
\end{equation}
Furthermore, in Algorithm 1, we are interested in the classical solutions of \eqref{model} and additionally assume that $\eta$ 
belongs to the Sobolev space $W^{1,1}_{loc}(\mathbb R_{+},\HH)$.

As evidenced by \eqref{model}, the function $h$ 
is of the form
\begin{equation} \label {frmh} 
h(t)=\sum\limits_{j=0}^M h_je^{-\rho_j(t-t_j)}\chi_{[t_j, \infty)}(t).
\end{equation}
Each term in the  above sum can model, for example, a new substance entering a chemical reaction. With this motivation, we will refer to these terms as catalysts; vector $h_j \in\mathcal{H}$ can be thought of as the content of the $j$-th catalyst,    $t_j \in \mathbb R_+$ -- as the time of its intake, and the exponent $\rho_j$ -- as its rate of decay in the system. 

Thus, our algorithms are designed to recover the unknown times \( t_j \) of catalyst intake, the rates of decay $\rho_j$, and certain ingredients in  the catalyst contents \( h_j \), which will be represented by a set of inner products $\{\langle h_j, g \rangle: g\in \mathcal G\}$ for some known set $\mathcal{G}\subset \mathcal{H}$.

To accomplish their purpose, our algorithms use different kinds of weak measurements of the solution of \eqref{model}. Although the designs of the measurements are different, in both cases, they are based on the same key idea. They combine information about the current state of the system with a prediction of its future state on condition that no new catalyst intake occurs in between. Consequently, thresholding certain linear combinations of the measurements allows us to  determine (with a reasonable degree of accuracy) if a catalyst with ingredients of interest has entered the system in a given time period. Once a new catalyst has been detected,  relatively   simple computations with the measurements provide estimates for its decay rate and content.  The technical difficulties of this research stem from the fact that we need to carefully select the thresholds and other parameters of the algorithms to provide meaningful error guarantees.

Before we proceed to the description of the algorithms let us add a few more fairly natural assumptions.

We assume that there is a known uniform bound $H$ on the mass of each catalyst:
\begin{equation}\label{Bddh}
    \sup_j\|h_j\|\le H.
\end{equation}

The set $\G$, which we choose to detect the catalyst ingredients, is likewise selected to be uniformly bounded: 
\[R = \sup_{g\in \G} \|g\| < \infty.\]

We also assume that intake times of consecutive catalysts  are separated:
\begin {equation} \label{tsep}
t_{j+1}-t_j\ge 4\beta +D,
\end{equation}
where $\beta$ is the time-step parameter chosen by the user of the algorithms and $D>0$ is a known parameter of the system. Clearly, a bigger value of $D$ leads to a smaller  effect of previous catalysts on the system at time of the new catalyst intake, which results in better error guarantees. 

Additionally, we assume that we know some positive bounds $\check{\rho}$ and  $\hat{\rho}$ on the minimal  and maximal  possible rates of catalyst decay: $\rho_j\in[\check{\rho}, \hat{\rho}].$

\subsection{Description of Algorithm 1.}

This algorithm relies on a combination of discrete samples of the solution of \eqref{model} taken at times \(n\beta\) and weighted average samples over the periods of time in between the discrete samples. Specifically, Algorithm 1 uses the following numbers as inputs:
\begin{equation}\label{measurement1}
    \mathfrak m_{n}(g,\beta)=\langle u((n+1)\beta),\frac{g}{\beta}\rangle-\langle u(n\beta),\frac{g}{\beta}\rangle-\int_{n\beta}^{(n+1)\beta}\langle u(t),A^{*}\frac{g}{\beta}\rangle dt+\nu_n, \quad g \in \G,
\end{equation}
where $\G\subseteq D(A^*)$ and  
 $\nu_n$ represents bounded  additive noise: 
\begin{equation}\label{noiseestimate}
|\nu_n|\le \sigma.
\end{equation} 


For determining the rates of decay,
Algorithm 1 also uses a second set of measurements at a finer time scale with the step size $\tilde \beta=\frac \beta N$: 
\begin{equation}\label{measurement2}
  \mathfrak s_{n}(g, \widetilde{\beta})=\langle u(n\beta+\widetilde{\beta}),\frac{g}{\widetilde{\beta}\beta}\rangle-\langle u(n\beta), \frac{g}{\widetilde{\beta}\beta}\rangle-\langle u(n\beta), A^*\frac{g}{\beta}\rangle+\widetilde{\nu}_{n},  \quad g \in \G,
\end{equation} where $\widetilde{\nu}_{n}$ is the measurement noise satisfying $|\widetilde{\nu}_{n}|\le \sigma$ and $N\in \mathbb N$ is a parameter chosen by the user that influences the accuracy of recovery.

\begin{remark}\label{SNRremark}
 We note that in this setting the upper bound $\sigma$  on measurement noises is assumed to be independent on the time-step parameters $\beta$ and $\widetilde\beta$. The noiseless part of the measurements, however, is amplified as $\beta$ and $\widetilde\beta$ decrease. In other words, we assume that measurement devices are designed to provide a higher signal-to-noise ratio when the time step $\beta$ (and/or $\widetilde\beta$) is smaller. We justify the practicality of this setup by stipulating that expending the same amount of energy over a shorter period of time should indeed lead to more accurate measurements by a device.    
\end{remark}

\begin{remark}\label{needclass}
    Algorithm 1 deals with the classical solution of 
\eqref{model}. To guarantee its existence (see Section \ref{stoolkit}), we make assumptions on the initial condition $u_0$ to belong to the domain $D(A)$ of the generator $A$ and the background source $\eta$ to belong to the Sobolev space $W^{1,1}_{loc}$. Working with the classical solution is necessary to effectively use the measurements 
\eqref{measurement2} as approximations of samples of the solution's derivative $\dot u$.
\end{remark}

We are now ready to provide a pseudo-code for Algorithm 1. 

\begin{tabular}{rp{13cm}}
\toprule
\multicolumn{2}{p{13cm}}{\textbf{Algorithm 1.} Pseudo-code for approximating the intake time, content ingredients  and decay rate of catalysts.
}\\
\midrule
1:& \textbf{Input:} Measurements: $\mathfrak m_{i}(g,\beta)$, $\mathfrak s_{i}(g,\frac{\beta}{N});$ Parameters: $K\ge 1$, $N\in\mathbb N$\\
2:& \textbf{Set} thresholds: $ Q(g,\beta)=K\widetilde{Q}(g,\beta)$,  $g\in\G$, using \eqref{Thresh} below\\
3:& Compute $\Delta_i(g)=\mathfrak m_{i}(g,\beta)-\mathfrak m_{i-1}(g,\beta)$\\
4:& \textbf{For} $g\in\G$ \textbf{do}\\
5:& \quad $i=1$\\
6:& \quad\textbf{while} $i\beta< \mathcal T$\\
7:& \quad\quad\textbf{if} $\Delta_i(g)>Q(g,\beta)$ \text{and} $|\mathfrak m_{i+1}-\mathfrak m_{i-2}|>\widetilde{Q}(g,\beta)$ \textbf{then}\\
8:& \quad\quad\quad$\mathfrak f_j(g):=\mathfrak m_{i+1}-\mathfrak m_{i-2}$\\
9:& \quad\quad\quad$ \estime_j(g) :=i\beta$\\
10:& \quad\quad\quad$i=i+3+\lfloor \frac{D}{\beta} \rfloor$\\
11:& \quad\quad\textbf{else}\\
12:& \quad\quad\quad $i=i+1$\\

13:& Set $\mathfrak t_j=\min\limits_{g\in\G} \mathfrak t_j(g)$\\
14:& Set $\widetilde{g}=\text{argmin}_{\{g\in\G: \mathfrak f_j(g)\ne0\}}\|g\|$\\
15:& Compute $\widetilde\Delta_i=\mathfrak s_{i+1}(\widetilde{g},\frac{\beta}{N})-\mathfrak s_{i}(\widetilde{g},\frac{\beta}{N})$\\
16:& Compute $R:=\left| \frac{\widetilde\Delta_{i+1}-\widetilde\Delta_{i-2}}{\mathfrak f_j(\widetilde{g})}\right|$\\
17: & \textbf{if} $R<\check \rho$ \\

& \textbf{then}  $\bar{\rho}_j:= \check \rho$ \\
18: & \textbf{else if} $R>\hat \rho$ \\

& \quad\textbf{then}  $\bar{\rho}_j:= \hat \rho$ \\
19:& \quad   $\bar{\rho}_j:= R$\\
20:& \textbf{Return} $\estime_j,$ $\bar{\rho}_j$ and $\mathfrak f_j(g)$ for all $g\in\G$.\\
\bottomrule
\end{tabular}

\subsection{Description of Algorithm 2.}


Just as the previous algorithm, this one also focuses on detecting catalysts one by one. In this case, however, we wish to make a more explicit use of this fact. To this end, we lump the previous catalysts together with the background source and write  
\begin{equation} \label{Reducedf}
F(t) = h_je^{\rho_j(t_j - t)}\chi_{[t_j,\infty)}(t) + \tilde{\eta}(t)
\end{equation}
where 
\begin{equation} \label{backTilde}
\tilde{\eta}(t) = \sum_{i < j} h_ie^{\rho_i(t_i - t)}\chi_{[t_i,\infty)}(t) + \eta(t).
\end{equation}

Algorithm 2 is designed to detect a new catalyst entering the system after the time $\ell_0\beta$, where $\ell_0 \in \mathbb N$ is chosen by the user of the algorithm.  
For each time-step $\ell\beta > \ell_0\beta > t_{j-1}$, the function $\tilde{\eta}$ in \eqref{backTilde} is Lipschitz on $[\ell\beta,\infty)$ with (local) Lipschitz constant estimated by 
\[
L_\ell = L + \frac{H\hat{\rho}e^{(\ell_{0} - \ell)\beta}}{1 - e^{-\check{\rho}(4\beta + D)}} \le  L + \frac{H\hat{\rho}}{1 - e^{-\check{\rho} D}},
\]
where $L$ is, as before, the (global) Lipschitz constant of the original background source $\eta$ (see Appendix for the proof). 

Algorithm 2 uses a single set of weighted average samples of the solution of \eqref{model}:
\begin{equation}\label{eq_abstract_m}
m_{s,\ell}(g,\beta) = \int_{\ell\beta}^{(\ell + 1)\beta} e^{-st} \ang{u(t),(\overline{s}I-A^*)\frac{g}{\beta^2}}\; dt + \nu_{s,\ell}, \ \ell \in \Z,
\end{equation}
where $ s=\frac{2\pi ik}{\beta} \in \mathbb C$ for some user-chosen parameter $k\in \Z\setminus\{0\}$ and the term $\nu_{s,\ell}$ represents the measurement noise, which is assumed to satisfy $|\nu_{s,\ell}|\le \sigma$ for some known $\sigma \ge 0$. The typical choice of $k$ will be $k=1$. Once again, the noise bound $\sigma$ is assumed to be independent of $\beta$, leading to higher measurement SNR for smaller time-step sizes (see Remark \ref{SNRremark}).

We are now ready to provide a pseudo-code for Algorithm 2, which uses the values $\Delta_{s,\ell}(g,\beta) = m_{s,\ell}(g,\beta) - m_{0,\ell}(g,\beta)$ as its inputs.
\medskip

\begin{tabular}{rp{13cm}}
\toprule
\multicolumn{2}{p{13cm}}{\textbf{Algorithm 2.} Pseudo-code for approximating the data $(\rho_j,t_j, \ang{h_j,g})$.}\\
\midrule
1:& \textbf{Input:} Measurements $\Delta_{s,\ell}(g,\beta)$ for   $g\in\mathcal G$, initial time-step $\ell_0$, parameter value $s = 2\pi i k/\beta$.\\
2:& $\ell=\ell_0$\\
3:&\textbf{while} $\ell\beta<\mathcal T$\\
4:& \quad\textbf{update} \[L_\ell = L + \frac{H\hat{\rho}e^{(\ell_{0} - (\ell-2))\beta}}{1 - e^{-\check{\rho}(4\beta+D)}} \] \\
5:& \quad\textbf{define} $Q_n(g)= \frac{4n}{\pi}L_\ell\|g\| + 4\sigma$, $n\in\{1,2,3\}$ \\
6:& \quad\textbf{if} $\sup\limits_{g\in \G} \big(|\Delta_{s,\ell}(g,\beta) - \Delta_{s,\ell-1}(g,\beta)| - Q_1(g) - \hat{\rho}H\|g\|\big) > 0$\\

&\quad\textbf{then}\\
7:& \quad\quad \textbf{return} $\estime_j := \ell\beta$\\
8:& \quad\quad\textbf{for} $g \in \mathcal{G} $ \textbf{do} \\
9:& \quad\quad\quad \textbf{if} $|\Delta_{s,\ell+1}(g,\beta) - \Delta_{s,\ell-2}(g,\beta)| > Q_3(g)$ \\
&\quad\quad\quad \textbf{then}\\
10: & \quad\quad\quad\quad\textbf{define} $R=\frac{\Delta_{s, \ell + 2}(g,\beta) - \Delta_{s,\ell+1}(g,\beta)}{\beta(\Delta_{s,\ell+1}(g,\beta)-\Delta_{s,\ell-2}(g,\beta))}$\\
11: & \quad\quad\quad\quad\textbf{if} $R<\check \rho$ \\

& \quad\quad\quad\quad\textbf{then} \textbf{return} $\tilde{\rho}_j:= \check \rho$ \\
12: & \quad\quad\quad\quad\textbf{else if} $R>\hat \rho$ \\

& \quad\quad\quad\quad\quad\quad\textbf{then} \textbf{return} $\tilde{\rho}_j:= \hat \rho$ \\
13:& \quad\quad\quad\quad\quad\quad \textbf{else return}  $\tilde{\rho}_j:= R$\\
14:& \quad\quad\quad\quad \textbf{return} \[\mathfrak f_j(g) :=  \frac{\Tilde{\rho}_j(\tilde{\rho}_j+s)\beta^2}{s\cdot(e^{-2\Tilde{\rho}_j\beta} - e^{-\Tilde{\rho}_j\beta})} (\Delta_{s,\ell+1}(g,\beta)-\Delta_{s,\ell-2}(g,\beta)) \]\\
15:& \quad\quad\quad \textbf{else} \textbf{return} $\mathfrak f_j(g) := 0$\\
16:& \quad\quad $\ell = \ell + 5\beta + \lfloor D/\beta \rfloor$ \\
17:& \quad\textbf{else}\\ 
18:& \quad\quad $\ell = \ell+1$ \\
19:& \textbf{Output:} Data $(\Tilde{\rho}_j,\estime_j,\mathfrak f_j(g))$ for all $g\in\mathcal S$.\\ 
\bottomrule
\end{tabular}

\section{Derivation of algorithms and their guaranties}\label{sec4}
In this section, we present theorems outlining performance guarantees for Algorithms 1 and 2. The derivation of the algorithms will become apparent in the process of proving the theorems.  


\subsection{Approximation Bounds for the Source Recovery by Algorithm 1}

We begin with a theorem that collects theoretical error guarantees for Algorithm 1.

\begin{theorem}\label{thm1}
Assume $u_0\in D(A).$ Let $\beta$ be the sampling time step and, given a parameter $D>0$, assume a minimal distance $t_{j+1}-t_j\ge D+4\beta$ between any two  consecutive times 
$t_{j+1}$ and $t_j$ of catalyst intake. Assume also that $\sup_j\|h_j\|\le H,$   $\sup_{g \in \G,}\|g\|\le R$, and $\rho_j\in [\check{\rho}, \hat{\rho}]$ for some known bounds $\hat{\rho} \ge \check{\rho} >0$.
 Let $\mathfrak t_j$, $\bar \rho_j$, and $\mathfrak f_j(g)$ be the outputs of Algorithm 1. Then $|\mathfrak t_j-t_j|\le\beta$ and 
\begin{equation} \label{Esthj}
|\mathfrak f_j(g)-\langle h_j,g\rangle|\le \widetilde{Q}(g,\beta)+v_2(h_j,g,\beta)+\frac{HR}{e^{\check{\rho} D}-1}(1-e^{-3\hat{\rho}\beta})+3L\beta\|g\|+2\sigma+2Q(g,\beta)
\end{equation}
where $Q(g,\beta)=K\widetilde{Q}(g,\beta)$ 
 is the threshold in Algorithm 1 given by \eqref {Thresh}  
and $v_2(h_j,g,\beta)$ -- by \eqref {vj} below. 

Additionally, if the output $\mathfrak f_j(g)$ satisfies $|\mathfrak f_j(g)| > \widetilde Q(g,\beta)$,   then the relative error of the decay rate $\rho_j$ is bounded by
    \begin{equation}\label{Estrj}
    \left|\frac{\rho_j-\bar{\rho}_j}{\rho_j}\right|\le\frac{1}{{|\mathfrak f_j(g)|}}\left(\frac{HR}{e^{\check{\rho} D}-1}(1-e^{-3\hat{\rho}\beta})\frac{\hat{\rho}-\check{\rho}}{\check{\rho}}+L\|g\|(\frac{2}{\check{\rho}}+3\beta)+\frac{\epsilon(\widetilde{\beta})}{\check{\rho}}+4\frac{\sigma}{\check{\rho}}+2\sigma\right),
    \end{equation}
where $\widetilde{\beta}=\frac{\beta}{N}$ for a user-chosen parameter $N\in\mathbb N$ and $\epsilon({\widetilde{\beta}})$ represents the approximation error when estimating $u_t(t)$ at some given time. Moreover, we have $\epsilon(\widetilde{\beta})\to 0$ as $N\to\infty.$
\end{theorem}
\medskip

The following remark is intended to clarify the quality of the estimates in the above theorem.

\begin {remark}${}$
\begin {enumerate}
\item When $\beta$ is sufficiently small, the error in \eqref{Esthj} will be smaller than $7\sigma$.
   
\item In the ideal scenario (where there is no background source and no measurement noise), the estimate in \eqref{Estrj} reduces to
    \[
    \left|\frac{\rho_j-\bar{\rho}_j}{\rho_j}\right|\le\frac{1}{|\mathfrak f_j(g)|}\left(\alpha\frac{\hat{\rho}-\check{\rho}}{\check{\rho}}\mathbf{e}(3\beta)+\frac{\epsilon(\widetilde{\beta})}{\check{\rho}}\right),\]
    where \( \mathbf e(t) = 1 - e^{-\hat\rho t} \) and $\alpha=\frac{1}{e^{\check{\rho} D}-1}HR.$
    
    \item The term $\frac{L}{\check\rho}$ in \eqref{Estrj} reflects the ability of Algorithm 1 to distinguish between the background source term and the catalysts $h_je^{-\rho_j(\cdot-t_j)}\chi_{[t_j, \infty)}$, $j=1,2,\ldots$. The relative error improves as $\frac{L}{\check\rho}$ decreases. On the other hand, when $L$ is close to $\check\rho$, the background source and the catalysts behave similarly and are difficult to distinguish.
\end{enumerate}
\end{remark}

\subsection {Details of Algorithm 1}

To detect the time of a catalyst intake in the presence of measurement noise and the background source, Algorithm 1 utilizes thresholds denoted by \(Q(g,\beta)\) for $g\in \mathcal G$ and $\beta > 0$. They are defined by
\begin{equation} \label{Thresh}
Q(g,\beta):=K\widetilde{Q}(g,\beta)=K\left((\alpha+H\|g\|)\mathbf{e}(\beta)+L\beta\|g\|+2\sigma\right),
\end{equation}
where  $\alpha=\frac{1}{e^{\check{\rho} D}-1}HR,$ \( \mathbf e(t) = 1 - e^{-\hat\rho t},\) and \( K\ge 1 \) is a  fixed parameter chosen by the user of the algorithm. As before, \( L \) is the Lipschitz constant of the background source \( \eta \) as given in \eqref{LipBackgrnd}, \( \sigma \) represents the upper bound for the sampling noise in \eqref{noiseestimate}, and \( \beta \) is the time  step size in sampling scheme. 

To see how the threshold is derived, we employ the following notation:
\( f_n \in \HH \) and \( \tau_n \in [n\beta, (n+1)\beta) \) for each \( n \) as follows:
\begin{equation}\label{f_n}
\left\{
\begin{array}{lll}
f_n = h_j, &\tau_n = t_j, & \text{if the \( j \)-th intake time  occurred in } [n\beta, (n+1)\beta);\\
f_n = 0, &\tau_n = n\beta, & \text{if no new source term emerged in } [n\beta, (n+1)\beta). 
\end{array}
\right.
\end{equation}
Using \eqref{f_n} and \eqref {model}, we rewrite the measurements in \eqref{measurement1} as
\begin{equation} \label {noiseEq}
    \begin{split}
    \mathfrak m_{n}(g,\beta)
        &=\sum\limits_{\tau_k<n\beta} \langle f_k,g\rangle e^{-\rho_k(n\beta-\tau_k)}\frac{e^{-\rho_k\beta}-1}{-\rho_k\beta}+ \langle f_n,g\rangle\frac{e^{-\rho_n((n+1)\beta-\tau_n)}-1}{-\rho_n\beta}\\
        &\quad+\int_0^{\beta}\langle\eta(n\beta+t),\frac{g}{\beta}\rangle dt+\nu_n.
    \end{split}
\end{equation}

To estimate the catalyst intake time $t_j$, we use the quantity $\Delta_n(g) = \mathfrak m_{n} - \mathfrak m_{n-1}$ 
rather than $\mathfrak m_{n}$ alone. Considering this difference allows us to mitigate the effect of the background source on the process. Indeed, if there is no new catalyst entering the system during the time interval $[n\beta, (n+1)\beta)$ (i.e., $f_n=0$ as in \eqref{f_n}), then
\begin{equation*}
    \begin{split}
\Delta_n(g)&=\mathfrak m_{n}-\mathfrak m_{n-1}\\
 &=\sum\limits_{\tau_k<(n-1)\beta} \langle f_k,g\rangle \frac{e^{-\rho_k\beta}-1}{-\rho_k\beta}(e^{-\rho_k(n\beta-\tau_k)}-e^{-\rho_k((n-1)\beta-\tau_k)})\\
 &\quad+\int_0^{\beta}\langle\eta(n\beta+t)-\eta((n-1)\beta+t),\frac{g}{\beta}\rangle dt+\nu_n-\nu_{n-1}\\
 &=\sum\limits_{\tau_k<(n-1)\beta} \langle f_k,g\rangle \frac{e^{-\rho_k\beta}-1}{-\rho_k\beta}e^{-\rho_k((n-1)\beta-\tau_k)}(e^{-\rho_k\beta}-1)\\
 &\quad+\int_0^{\beta}\langle\eta(n\beta+t)-\eta((n-1)\beta+t),\frac{g}{\beta}\rangle dt+\nu_n-\nu_{n-1}.
\end{split}
\end{equation*}

Using the intake separation assumption \eqref {tsep}, we only analyze $|\Delta_n(g)|$ for intervals $[n\beta, (n+1)\beta)$ that are at least $(3+\lfloor \frac{D}{\beta}\rfloor)\beta$ behind the previously estimated intake time $t_{j-1}$. Because of this fact and the assumption that no  new catalyst enters in $[n\beta, (n+1)\beta)$, we obtain the following estimate for $|\Delta_n(g)|$:

\begin{equation}\label{delta}
    \begin{split} 
|\Delta_n(g)|
&\le\left|\sum\limits_{\tau_k<(n-1)\beta} \langle f_k,g\rangle \frac{e^{-\rho_k\beta}-1}{-\rho_k\beta}e^{-\rho_k((n-1)\beta-\tau_k)}(e^{-\rho_k\beta}-1)\right|\\
&\quad+\left|\int_0^{\beta}\langle\eta(n\beta+t)-\eta((n-1)\beta+t),\frac{g}{\beta}\rangle dt\right|+|\nu_n-\nu_{n-1}|\\
&\le \sum\limits_{\tau_k<(n-1)\beta}|\langle f_k,g\rangle e^{-\rho_k((n-1)\beta-\tau_k)}|\mathbf{e}(\beta)+\|\eta(n\beta+t)-\eta((n-1)\beta+t)\|\|g\|+2\sigma\\
&\le \sum\limits_{k=1}^{\infty}\|f_k\|\|g\|e^{-\check{\rho} kD}\mathbf{e}(\beta)+L\beta\|g\|+2\sigma\\
&\le\frac{e^{-\check{\rho} D}}{1-e^{-\check{\rho} D}}\mathbf{e}(\beta)HR+L\beta\|g\|+2\sigma\\
&\le \frac{1}{e^{\check{\rho} D}-1}\mathbf{e}(\beta)HR+L\beta\|g\|+2\sigma\\
&=\alpha\mathbf{e}(\beta)+L\beta\|g\|+2\sigma
\end{split}
\end{equation}
where $\alpha=\frac{1}{e^{\check{\rho} D}-1}HR$ and $\mathbf{e}(\beta)=1-e^{-\hat{\rho}\beta}.$
The last term of the set of inequalities above partially justifies our choice of the thresholds $Q(g,\beta)$ in \eqref{Thresh}. 

On the other hand, in the event of a new catalyst entering the system within the time interval $[n\beta, (n+1)\beta)$ (i.e., $f_n=h_j \neq 0$), it's crucial that if it cannot be detected by $\Delta_n(g)$ or $\Delta_{n+1}(g)$, it stays undetectable in all subsequent measurements. Otherwise, it could lead to a substantial error in the estimation of the intake time. Thus, we need to ensure $\Delta_i(g)$ remain sufficiently small for $i\ge n+2$ until  the following intake time $t_{j+1}.$ We estimate 
\begin{equation}\label{tildethreshold}
\begin{split}
 |\Delta_{i}(g)|&=\left|\sum\limits_{\tau_k<(n-1)\beta} \langle f_k,g\rangle \frac{e^{-\rho_k\beta}-1}{-\rho_k\beta}e^{-\rho_k((i-1)\beta-\tau_k)}(e^{-\rho_k\beta}-1)\right|\\
 &\quad+\left|\langle h_j,g\rangle \frac{e^{-\rho_k\beta}-1}{-\rho_k\beta}e^{-\rho_k((i-1)\beta-t_j)}(e^{-\rho_k\beta}-1)\right|\\
 &\quad+\left|\int_0^{\beta}\langle\eta(i\beta+t)-\eta((i-1)\beta+t),\frac{g}{\beta}\rangle dt\right|+\left|\nu_i-\nu_{i-1}\right|   \\
 &\le \alpha\mathbf{e}(\beta)+\|f_j\|\|g\|(1-e^{-\hat{\rho}\beta})+L\beta\|g\|+2\sigma\\
 &\le\alpha\mathbf{e}(\beta)+H\|g\|\mathbf{e}(\beta)+L\beta\|g\|+2\sigma:=\widetilde{Q}(g,
 \beta).
\end{split}
\end{equation}
The above set of inequalities serves as justification for our selection of the thresholds $\widetilde Q(g,\beta)$ and completely justifies \eqref{Thresh}.

 If the value of $|\Delta_n(g)|$ 
 is larger than the threshold $Q$, then we know that a new catalyst entered the system in the time interval $[(n-1)\beta, (n+1)\beta)$. The occurrence of this case doesn't guarantee, however, that $t_j$ is within the interval $[n\beta,(n+1)\beta)$. It is possible that $t_j$ belongs to the interval $[(n-1)\beta, n\beta)$ with $t_j$ sufficiently close to $n\beta$. Nevertheless, in both scenarios, Algorithm 1 records estimates $\mathfrak t_j(g)$ for the intake time $t_j$ as $n\beta.$ There is an extra intricacy that occurs due to the fact that there might be instances where $\mathfrak t_j(g_{i_1})=n\beta$ and $\mathfrak t_j(g_{i_2})=(n+1)\beta$ for distinct $g_{i_1}, g_{i_2}\in\G$. To handle this, Algorithm 1  defines $\mathfrak t_j:=\min\limits_{g\in\G}\mathfrak t_j(g)$. As a result, we get the error estimate $|t_j-\mathfrak t_j|\leq \beta.$ 
 
 Next,  We use the samples before and after the interval of intake to estimate the parameter values of the catalyst.
 From the above discussion, it is clear that given any $g\in\G$, the detection of $h_j$ is only possible through either $\Delta_n(g)$ or $\Delta_{n+1}(g)$. 
 In particular, if the values of  $|\Delta_n(g)|$ and $|\Delta_{n+1}(g)|$ are smaller than or equal to the threshold $Q$ for each $g\in\G$, the output of Algorithm 1 will indicate that no new catalyst entered the system during the period. It may, however, be the case that 
 the value $|\langle h_j, g\rangle|$ happened to be too small for the ingredient to be detected. 
 Thus, to establish estimate \eqref{Esthj} for the error of recovery of $\langle h_j, g\rangle$ with $t_j\in[n\beta, (n+1)\beta)$, we need to distinguish between the following three outcomes of Algorithm 1:

\noindent \textbf{Cases:}\label{Cases}
\begin{enumerate}
    \item \label{c1} $|\Delta_n(g)|>Q,$ so that Algorithm 1 sets $\mathfrak t_j(g)=n\beta.$
    \item \label{c2} $|\Delta_n(g)|\le Q$ and $|\Delta_{n+1}(g)|>Q$, 
    so that Algorithm 1 sets
    $\mathfrak t_j(g)=(n+1)\beta.$
    \item \label{c3} $|\Delta_n(g)|\le Q$ and $|\Delta_{n+1}(g)|\le Q$, so that Algorithm 1 does not detect the intake 
    (for example if $\|h_j\|$ is too small).
\end{enumerate}

From the point of view of the algorithm itself, however, Cases (1) and (2) are indistinguishable. Moreover, for the estimate \eqref{Estrj} to be meaningful, the estimates $\mathfrak f_j(g)$ of the quantities $ \langle h_j, g \rangle$, $g\in\G$, should have a lower bound, which Algorithm 1 sets as $\widetilde{Q}(g,\beta)$ defined in \eqref {tildethreshold}. Thus, Algorithm 1 defines $\mathfrak f_j(g)$ by

\begin{equation*}
\mathfrak f_j(g):=
	\begin{cases}
	\mathfrak{m}_{n+1}(g,\beta) - \mathfrak{m}_{n-2}(g,\beta),\ \text{Case (1) with } \mathfrak{m}_{n+1}(g,\beta) - \mathfrak{m}_{n-2}(g,\beta)>\widetilde{Q}(g,\beta),\\
    \mathfrak{m}_{n+2}(g,\beta) - \mathfrak{m}_{n-1}(g,\beta),\ \text{Case (2) with } \mathfrak{m}_{n+2}(g,\beta) - \mathfrak{m}_{n-1}(g,\beta)>\widetilde{Q}(g,\beta),\\
	0,\  \text{otherwise}.
	\end{cases}
\end{equation*}
Once $\mathfrak f_j(g) \neq 0$, Algorithm 1 chooses $\widetilde g$ with the minimal norm among all such $g\in\G$ 
and approximates the derivative of $u$ near the intake time using the measurements $\mathfrak{s}_{i}(\widetilde g,\frac{\beta}{N})$. 
More precisely, Algorithm 1 uses the quantity $\widetilde\Delta_i=\mathfrak s_{i+1}(\widetilde{g},\frac{\beta}{N})-\mathfrak s_{i}(\widetilde{g},\frac{\beta}{N})$ and estimates    
the decay rate $\rho_j$  by $\bar{\rho}_j:=\left| \frac{\widetilde\Delta_{i+1}-\widetilde\Delta_{i-2}}{\mathfrak f_j(\widetilde{g})}\right|$.


The reasons behind choosing the above estimates for $ \langle h_j, g \rangle$ and $\rho_j$ will be revealed in the following subsection.

\subsection {Proof of Theorem \ref {thm1}}
The proof of Theorem \ref {thm1} has two major parts. In Subsection \ref {sec1}, we estimate the quantity $ \langle h_j, g \rangle$, whereas subsection \ref {sec2} covers the approximation of the decay rates $\rho_j$. 

\subsubsection{Derivation of \eqref {Esthj} in Theorem \ref {thm1}}\label{sec1}
We start with the following Lemma.
\begin{lemma}\label{lem1} 
Assume that $t_j\in[n\beta, (n+1)\beta)$ and 
\begin{equation} \label {vj}
v_k(h_j,g,\beta)=\left|\langle h_j,g\rangle e^{-\rho_j((n+k)\beta-t_j)}\frac{e^{-\rho_j\beta}-1}{-\rho_j\beta}-\langle h_j,g\rangle\right|,\quad k=1,2.
\end{equation}
Then
$v_k(h_j,g,\beta)\to 0$ as $\beta\to 0.$
\end{lemma}
\begin{proof}
Observe that due to $0<\frac{e^{-\rho_j\beta}-1}{-\rho_j\beta}<1$ we have
\begin{equation}\label{tempeq}
  \begin{split}
v_k(h_j,g,\beta)&=|\langle h_j,g\rangle|
\left(1-e^{-\rho_j((n+k)\beta-t_j)}\frac{e^{-\rho_j\beta}-1}{-\rho_j\beta}\right)
\\
&\le\|h_j\|\|g\|\left(1-e^{-\rho_j((n+k)\beta-t_j)}\frac{e^{-\rho_j\beta}-1}{-\rho_j\beta}\right).
\end{split}  
\end{equation}
Additionally, $0<(n+k)\beta-t_j\le 2\beta$ 
and 
$\frac{e^{-\rho_j\beta}-1}{-\rho_j\beta}\to 1$ as $\beta\to 0.$ 
Therefore,
\[
0<1-e^{-\rho_j((n+k)\beta-t_j)}\frac{e^{-\rho_j\beta}-1}{-\rho_j\beta}\le1-e^{-2\rho_j\beta}\frac{e^{-\rho_j\beta}-1}{-\rho_j\beta}\to 0 \text{ as }\beta\to 0.
\]
Hence $v_k(h_j,g,\beta)\to 0$ as $\beta\to 0.$  
\end{proof}

\begin{remark}\label{easeyrem}
    In view of the equality in \eqref{tempeq}, the
    inequality  $(n+2)\beta-t_j> (n+1)\beta-t_j$ implies that     $v_1(h_j, g, \beta)\leq v_2(h_j, g, \beta)$.
\end{remark}

Next, we will analyze the three cases described in Subsection  \ref{Cases}. Recall that we assumed
 that the $j$-th intake time $t_j$ occurred within the time interval $[n\beta, (n+1)\beta)$ (i.e., $f_n=h_j$) and $Q=K\widetilde{Q}(g,\beta)=K((\alpha+H\|g\|)\mathbf{e}(\beta)+L\beta\|g\|+2\sigma)$, where $K\ge 1$, is the threshold given by \eqref{Thresh}.  

For Case (1), we obtain the following result.
\begin{lemma}\label{case1}
Assume $|\Delta_n(g)|=|\mathfrak m_n-\mathfrak m_{n-1}|>Q.$ If $|\mathfrak m_{n+1}-\mathfrak m_{n-2}|>\widetilde{Q},$ $\langle h_j,g\rangle$ can be approximated by $\mathfrak f_j(g)=\mathfrak m_{n+1}-\mathfrak m_{n-2}$ and we have 
\begin{equation}
\begin{split} 
|\mathfrak f_j(g)-\langle h_j,g\rangle|\le v_1(h_j,g,\beta)+\alpha\mathbf{e}(3\beta)+3L\beta\|g\|+2\sigma,
\end{split}
\end{equation}
where $v_1(h_j,g,\beta)$ is given in Lemma \ref{lem1}. Otherwise, we set $\mathfrak f_j(g)=0$ and
\begin{equation}
  \begin{split}
|\langle h_j,g\rangle|\le \widetilde{Q}(g,\beta)+v_1(h_j,g,\beta)+\alpha\mathbf{e}(3\beta)+3L\beta\|g\|+2\sigma.
  \end{split}
\end{equation}
\end{lemma}

\begin{proof}
Observe that using \eqref {noiseEq} we have 
\begin{equation}\label{nequ}
\begin{split} 
\mathfrak m_{n+1} -\mathfrak m_{n-2} 
&=\sum\limits_{\tau_k<(n-2)\beta} \langle f_k,g\rangle \frac{e^{-\rho_k\beta}-1}{-\rho_k\beta}e^{-\rho_k((n-2)\beta-\tau_k)}(e^{-3\rho_k\beta}-1) \\
&+\langle h_j,g\rangle e^{-\rho_j((n+1)\beta-t_j)}\frac{e^{-\rho_j\beta}-1}{-\rho_j\beta}\\
&+\int_0^\beta\langle\eta((n+1)\beta+t)-\eta((n-2)\beta+t),\frac{g}{\beta}\rangle dt+\nu_{n+1}-\nu_{n-2}.
\end{split}
\end{equation}

Therefore, if $|\mathfrak m_{n+1}-\mathfrak m_{n-2}|>\widetilde{Q},$ then $\mathfrak f_j(g)=\mathfrak m_{n+1}-\mathfrak m_{n-2}$, and we get 
\begin{equation}\label{appro1}
\begin{split} 
|\mathfrak f_j(g)-\langle h_j,g\rangle|
&\le\left|\sum\limits_{\tau_k<(n-2)\beta} \langle f_k,g\rangle \frac{e^{-\rho_k\beta}-1}{-\rho_k\beta}e^{-\rho_k((n-2)\beta-\tau_k)}(e^{-3\rho_k\beta}-1)\right|\\
&\quad+\left|\langle h_j,g\rangle e^{-\rho_j((n+1)\beta-t_j)}\frac{e^{-\rho_j\beta}-1}{-\rho_j\beta}-\langle h_j,g\rangle\right|\\
&\quad+\left|\int_0^\beta\langle\eta((n+1)\beta+t)-\eta((n-2)\beta+t),\frac{g}{\beta}\rangle dt\right|+\left|\nu_{n+1}-\nu_{n-2}\right|\\
&\le v_1(h_j,g,\beta)+\alpha\mathbf{e}(3\beta)+3L\beta\|g\|+2\sigma,
\end{split}
\end{equation}
where $v_1(h_j,g,\beta)=\left|\langle h_j,g\rangle e^{-\rho_j((n+1)\beta-t_j)}\frac{e^{-\rho_j\beta}-1}{-\rho_j\beta}-\langle h_j,g\rangle\right|$ is given in Lemma \ref{lem1} and we estimated other terms as in \eqref{delta}.

If $|\mathfrak m_{n+1}-\mathfrak m_{n-2}|\le\widetilde{Q},$ then $\mathfrak f_j(g)=0$, and using \eqref{nequ} we get 
\begin{equation*}
  \begin{split}
|\langle h_j,g\rangle|&
\le|\mathfrak m_{n+1}-\mathfrak m_{n-2}|+|(\mathfrak m_{n+1}-\mathfrak m_{n-2})-\langle h_j,g\rangle|\\
&\le \widetilde{Q}(g,\beta)+v_1(h_j,g,\beta)+\alpha\mathbf{e}(3\beta)+3L\beta\|g\|+2\sigma.
  \end{split}
\end{equation*}
The lemma is proved.
\end{proof}

Through similar computations, we can establish a  lemma pertaining to Case (2).
\begin{lemma}\label{case2}
Assume $|\Delta_n(g)|=|\mathfrak m_n-\mathfrak m_{n-1}|\le Q$ and $|\Delta_{n+1}(g)|=|\mathfrak m_{n+1}-\mathfrak m_{n}|>Q.$ If $|\mathfrak m_{n+1}-\mathfrak m_{n-2}|>\widetilde{Q},$ $\langle h_j,g\rangle$ can be approximated by $\mathfrak f_j(g)=\mathfrak m_{n+2}-\mathfrak m_{n-1}$ and we have 
\begin{equation}
\begin{split} 
|\mathfrak f_j(g)-\langle h_j,g\rangle|\le v_2(h_j,g,\beta)+\alpha\mathbf{e}(3\beta)+3L\beta\|g\|+2\sigma,
\end{split}
\end{equation}
where $v_2(h_j,g,\beta)$ is given in Lemma \ref{lem1}. Otherwise, we set $\mathfrak f_j(g)=0$ and
\begin{equation}
  \begin{split}
|\langle h_j,g\rangle|\le \widetilde{Q}(g,\beta)+v_2(h_j,g,\beta)+\alpha\mathbf{e}(3\beta)+3L\beta\|g\|+2\sigma.
  \end{split}
\end{equation}
\end{lemma}

Finally, Case (3) is covered by the following lemma.

\begin{lemma}\label{case3}
If $|\Delta_n(g)|\le Q$ and $|\Delta_{n+1}(g)|\le Q,$ let $\mathfrak f_j(g)=0,$ then we have
\begin{equation}
    |\mathfrak f_j(g)-\langle h_j,g\rangle|\le v_1(h_j,g,\beta)+\alpha\mathbf{e}(2\beta)+2L\beta\|g\|+2\sigma+2Q(g,\beta).
\end{equation}
\end{lemma}

\begin{proof}
In this case, neither $|\Delta_n(g)|$ nor $|\Delta_{n+1}(g)|$ is capable of detecting $h_j$; thus, we can conclude that $\langle h_j,g\rangle$ is small and is of the order of the threshold. We set $\mathfrak f_j(g)$ to 0 in this case, the following expression gives the error:

\begin{equation}\label{appro3}
\begin{split} 
|\langle h_j,g\rangle|&\le|\langle h_j,g\rangle-(\mathfrak m_{n+1}-\mathfrak m_{n-1})|+|\mathfrak m_{n+1}-\mathfrak m_{n-1}|\\
&\le\left|\sum\limits_{\tau_k<(n-1)\beta} \langle f_k,g\rangle \frac{e^{-\rho_k\beta}-1}{-\rho_k\beta}e^{-\rho_k((n-1)\beta-\tau_k)}(e^{-2\rho_k\beta}-1)\right|\\
&\quad+\left|\langle h_j,g\rangle e^{-\rho_j((n+1)\beta-t_j)}\frac{e^{-\rho_j\beta}-1}{-\rho_j\beta}-\langle h_j,g\rangle\right|\\
&\quad+\left|\int_0^\beta\langle\eta((n+1)\beta+t)-\eta((n-1)\beta+t),\frac{g}{\beta}\rangle dt\right|+\left|\nu_{n+1}-\nu_{n-1}\right|\\
&\quad+|\mathfrak m_{n+1}-\mathfrak m_{n}|+|\mathfrak m_{n}-\mathfrak m_{n-1}|\\
&\le v_1(h_j,g,\beta)+\alpha\mathbf{e}(2\beta)+2L\beta\|g\|+2\sigma+2Q(g,\beta),
\end{split}
\end{equation}
where we made a similar estimate as in \eqref{appro1} and used the fact that $|\Delta_{n}| = |\mathfrak m_{n} - \mathfrak m_{n-1}| \le Q(g,\beta)$ and $|\Delta_{n+1}(g)| = |\mathfrak m_{n+1} - \mathfrak m_{n}| \le Q(g,\beta)$.
\end{proof}

Combining Lemmas \ref{lem1}, \ref{case1}, \ref{case2}, and \ref{case3} with Remark \ref{easeyrem}, we obtain \eqref{Esthj}.


\subsubsection{Recovery of Decay Rate.}\label{sec2}
Let us assume that an intake time $t_j$ has been determined by $\Delta_n(g)$ (i.e., $|\Delta_n(g)| > Q(g, \beta)$). According to Algorithm 1, $\langle h_j,g\rangle\approx\mathfrak f_j(g)=\mathfrak m_{n+1}-\mathfrak m_{n-2}$. In order to estimate the decay rate $\rho_j$, approximate values of $u_t((n-2)\beta),$ $u_t((n-1)\beta),$ $u_t((n+1)\beta)$, and $u_t((n+2)\beta)$ are used by Algorithm 1. As we can only obtain noisy measurements of $u(t),$ we use them to approximate $u_t(t)$; we expect the error of this approximation to be reasonable because $u\in C^1([0,\mathcal T],\HH)$. Consider $u_t(n\beta)$ for example. We will fix $g$ and $\beta$, and then we divide the time interval $[n\beta, (n+1)\beta]$ into $N$ subintervals of length $\widetilde{\beta}=\frac{\beta}{N}$ ($N\in\Z^+$) and consider $\frac{u(n\beta+\widetilde{\beta})-u(n\beta)}{\widetilde{\beta}}$ as an approximation of $u_t(n\beta).$ Observe that 
\[
\begin{split}
    \langle u_t(n\beta), \frac{g}{\beta}\rangle&=\langle\frac{u(n\beta+\widetilde{\beta})-u(n\beta)}{\widetilde{\beta}},  \frac{g}{\beta}\rangle-\epsilon_{n}(\widetilde{\beta})\\
    &=\langle u(n\beta+\widetilde{\beta})-u(n\beta), \frac{g}{\widetilde{\beta}\beta}\rangle-\epsilon_{n}(\widetilde{\beta})
\end{split}
\]
where $|\epsilon_{n}(\widetilde{\beta})|\to 0$ as $\widetilde{\beta}\to 0$ (i.e. $N\to\infty$). Therefore at $t=n\beta,$ \eqref{model} implies
\begin{equation} \label{DerApp}
\begin{split} 
&\quad\langle u(n\beta+\widetilde{\beta})-u(n\beta), \frac{g}{\widetilde{\beta}\beta}\rangle-\langle u(n\beta), A^*\frac{g}{\beta}\rangle\\
&=\sum\limits_{\tau_k\le n\beta}\langle f_k,g\rangle \frac{e^{-\rho_k(n\beta-\tau_k)}}{\beta}+\langle\eta(n\beta),\frac{g}{\beta}\rangle+\epsilon_{n}(\widetilde{\beta}).
\end{split}
\end{equation}


\medskip


Under the assumption that $|\Delta_n(g)| > Q(g, \beta)$, using the expression \eqref{DerApp}, we evaluate the quantities $\widetilde{\Delta}_{n+1}(g)$, $\widetilde{\Delta}_{n-2}(g)$ defined in Algorithm 1 and their difference from the measurements \eqref{measurement2}.

\begin{equation*}
\begin{split} 
\widetilde{\Delta}_{n+1}(g)&=\mathfrak s_{n+2}-\mathfrak s_{n+1}\\
&=\sum\limits_{\tau_k<(n+1)\beta}\langle f_k,g\rangle\frac{e^{-\rho_k((n+2)\beta-\tau_k)}-e^{-\rho_k((n+1)\beta-\tau_k)}}{\beta}+\langle\eta((n+2)\beta)-\eta((n+1)\beta),\frac{g}{\beta}\rangle\\
&\quad+\epsilon_{n+2}(\widetilde{\beta})-\epsilon_{n+1}(\widetilde{\beta})+\widetilde{\nu}_{n+2}-\widetilde{\nu}_{n+1}\\
&=\sum\limits_{\tau_k<(n+1)\beta}\langle f_k,g\rangle e^{-\rho_k((n+1)\beta-\tau_k)}\frac{e^{-\rho_k\beta}-1}{\beta}+\langle\eta((n+2)\beta)-\eta((n+1)\beta),\frac{g}{\beta}\rangle\\
&\quad+\epsilon_{n+2}(\widetilde{\beta})-\epsilon_{n+1}(\widetilde{\beta})+\widetilde{\nu}_{n+2}-\widetilde{\nu}_{n+1},
\end{split}
\end{equation*}

\begin{equation*}
    \begin{split} 
         \widetilde{\Delta}_{n-2}(g)&=\mathfrak s_{n-1}-\mathfrak s_{n-2}\\
         &=\sum\limits_{\tau_k<(n-2)\beta}\langle f_k,g\rangle e^{-\rho_k((n-2)\beta-\tau_k)}\frac{e^{-\rho_k\beta}-1}{\beta}+\langle\eta((n-1)\beta)-\eta((n-2)\beta),\frac{g}{\beta}\rangle\\
         &\quad+\epsilon_{n-1}(\widetilde{\beta})-\epsilon_{n-2}(\widetilde{\beta})+\widetilde{\nu}_{n-1}-\widetilde{\nu}_{n-2},
\end{split}
\end{equation*}
and
\begin{equation*}
    \begin{split} 
         \widetilde{\Delta}_{n+1}(g)-\widetilde{\Delta}_{n-2}(g)
         &=\sum\limits_{\tau_k<(n-2)\beta}\langle f_k,g\rangle \frac{e^{-\rho_k\beta}-1}{\beta}e^{-\rho_k((n-2)\beta-\tau_k)}(e^{-3\rho_k\beta}-1)\\
         &\quad+\langle h_j,g\rangle e^{-\rho_j((n+1)\beta-t_j)}\frac{e^{-\rho_j\beta}-1}{\beta}+w_1(\eta, \epsilon, \widetilde{\nu})
    \end{split}
\end{equation*}
where $w_1(\eta,\epsilon,\widetilde{\nu})=\langle\eta((n+2)\beta)-\eta((n+1)\beta),\frac{g}{\beta}\rangle-\langle\eta((n-1)\beta)+\eta((n-2)\beta),\frac{g}{\beta}\rangle+\epsilon_{n+2}(\widetilde{\beta})-\epsilon_{n+1}(\widetilde{\beta})-\epsilon_{n-1}(\widetilde{\beta})+\epsilon_{n-2}(\widetilde{\beta})+\widetilde{\nu}_{n+2}-\widetilde{\nu}_{n+1}-\widetilde{\nu}_{n-1}+\widetilde{\nu}_{n-2}.$

Recall that
\begin{equation} \label {Eqfj}
\begin{split} 
\mathfrak f_j(g)
&=\mathfrak m_{n+1}-\mathfrak m_{n-2}\\
&=\sum\limits_{\tau_k<(n-2)\beta} \langle f_k,g\rangle \frac{e^{-\rho_k\beta}-1}{-\rho_k\beta}e^{-\rho_k((n-2)\beta-\tau_k)}(e^{-3\rho_k\beta}-1)\\
&\quad+\langle h_j,g\rangle e^{-\rho_j((n+1)\beta-t_j)}\frac{e^{-\rho_j\beta}-1}{-\rho_j\beta}+w_2(\eta,\nu)
\end{split}
\end{equation}
where $w_2(\eta,\nu)=\int_0^\beta\langle\eta((n+1)\beta+t)-\eta((n-2)\beta+t),\frac{g}{\beta}\rangle dt+\nu_{n+1}-\nu_{n-2}.$

We use $\left|\frac{\widetilde{\Delta}_{n+1}(g)-\widetilde{\Delta}_{n-2}(g)}{\mathfrak{f}_j(g)}\right|$ as an approximation for $\rho_j$ and derive the following two estimates before calculating the relative error in our estimation of $\rho_j$:

\begin{equation}\label{es1}
    \begin{split} 
         &\quad\left|\sum\limits_{\tau_k<(n-2)\beta}\langle f_k,g\rangle\frac{e^{-\rho_k\beta}-1}{\beta} e^{-\rho_k((n-2)\beta-\tau_k)}(e^{-3\rho_k\beta}-1)\left(1-\frac{\rho_j}{\rho_k}\right)\right|\\
         &\le\sum\limits_{\tau_k<(n-2)\beta}\|f_k\|\|g\|\rho_ke^{-\rho_k((n-2)\beta-\tau_k)}\mathbf{e}(3\beta)\left|1-\frac{\rho_j}{\rho_k}\right|\\
         &\le\sum\limits_{\tau_k<(n-2)\beta}\|f_k\|\|g\|e^{-\rho_k((n-2)\beta-\tau_k)}\mathbf{e}(3\beta)|\rho_k-\rho_j|\\
         &\le\sum\limits_{\tau_k<(n-2)\beta}e^{-\check{\rho} kD}\|f_k\|\|g\|\mathbf{e}(3\beta)(\hat{\rho}-\check{\rho})\\
         &\le\frac{1}{e^{\check{\rho} D}-1}HR(\hat{\rho}-\check{\rho})\mathbf{e}(3\beta)\\
         &=\alpha(\hat{\rho}-\check{\rho})\mathbf{e}(3\beta),
    \end{split}
\end{equation}
where $\check{\rho}, \hat{\rho}$ are the lower and upper bound of all $\rho_j$ respectively, $\mathbf{e}(3\beta)=1-e^{3\hat{\rho}\beta},$ and $\alpha=\frac{1}{e^{\check{\rho} D}-1}HR$; 
 
\begin{equation}\label{es2}
\begin{split} 
&\quad\left|w_1(\eta,\epsilon, \widetilde{\nu})+\rho_jw_2(\eta,\nu)\right|\\
&\le\left|\langle\eta((n+2)\beta)-\eta((n+1)\beta),\frac{g}{\beta}\rangle\right|+\left|\langle\eta((n-1)\beta)-\eta((n-2)\beta),\frac{g}{\beta}\rangle\right|\\
&\quad+|\epsilon_{n+2}(\widetilde{\beta})-\epsilon_{n+1}(\widetilde{\beta})-\epsilon_{n-1}(\widetilde{\beta})+\epsilon_{n-2}(\widetilde{\beta})|+|\widetilde{\nu}_{n+2}-\widetilde{\nu}_{n+1}-\widetilde{\nu}_{n-1}+\widetilde{\nu}_{n-2}|\\
&\quad+\left|\rho_j\int_0^\beta\langle\eta((n+1)\beta+t)-\eta((n-2)\beta+t),\frac{g}{\beta}\rangle dt\right|+\rho_j|\nu_{n+1}-\nu_{n-2}| \\
&\le 2L\|g\|+\epsilon(\widetilde{\beta})+4\sigma+3\rho_jL\beta\|g\|+2\rho_j\sigma,
\end{split}
\end{equation}
where $L$ is the Lipschitz constant of the background source, $\sigma$ is the noise level of the measurements, and $\epsilon(\widetilde{\beta})=|\epsilon_{n+2}(\widetilde{\beta})-\epsilon_{n+1}(\widetilde{\beta})-\epsilon_{n-2}(\widetilde{\beta})+\epsilon_{n-3}(\widetilde{\beta})|.$

Finally, we let 
$\frac{\widetilde\Delta_{n+1}(g)-\widetilde\Delta_{n-2}(g)}{\mathfrak{f}_j(g)}\approx -\rho_j$ and present the estimation of the relative error of $\rho_j$ using \eqref{Eqfj}.
\begin{equation}\label{relativeerror}
\begin{split} 
&\quad\frac{\left|\rho_j+\frac{\widetilde\Delta_{n+1}(g)-\widetilde\Delta_{n-2}}{\mathfrak f_j(g)}\right|}{\rho_j}\\ 
&=
\frac{|\widetilde\Delta_{n+1}(g)-\widetilde\Delta_{n-2}(g)+\rho_j\mathfrak f_j(g)|}{\rho_j|\mathfrak f_j(g)|}\\
&=
\frac{\left|\sum\limits_{\tau_k<(n-2)\beta}\langle f_k,g\rangle\frac{e^{-\rho_k\beta}-1}{\beta} e^{-\rho_k((n-2)\beta-\tau_k)}(e^{-3\rho_k\beta}-1)(1-\frac{\rho_j}{\rho_k})+w_1(\eta,\epsilon,\widetilde{\nu})+\rho_j w_2(\eta,\nu)\right|}{\rho_j|\mathfrak f_j(g)|}\\
&\le
\frac{\left|\sum\limits_{\tau_k<(n-2)\beta}\langle f_k,g\rangle\frac{e^{-\rho_k\beta}-1}{\beta} e^{-\rho_k((n-2)\beta-\tau_k)}(e^{-3\rho_k\beta}-1)(1-\frac{\rho_j}{\rho_k})\right|+\left|w_1(\eta,\epsilon,\widetilde{\nu})+\rho_j w_2(\eta,\nu)\right|}{\rho_j|\mathfrak f_j(g)|}\\
&\le
\frac{1}{{|\mathfrak f_j(g)|}}\left(\alpha\frac{\hat{\rho}-\check{\rho}}{\rho_j}\mathbf{e}(3\beta)+2\frac{L}{\rho_j}\|g\|+\frac{\epsilon(\widetilde{\beta})}{\rho_j}+3L\beta\|g\|+4\frac{\sigma}{\rho_j}+2\sigma\right)\\
&\le 
\frac{1}{{|\mathfrak f_j(g)|}}\left(\alpha\frac{\hat{\rho}-\check{\rho}}{\check{\rho}}\mathbf{e}(3\beta)+L\|g\|(\frac{2}{\check{\rho}}+3\beta)+\frac{\epsilon(\widetilde{\beta})}{\check{\rho}}+4\frac{\sigma}{\check{\rho}}+2\sigma\right).
\end{split}
\end{equation}
The validity of the last inequality is ensured by \eqref{es1} and \eqref{es2}.

\subsection{Approximation Bounds for the Source Recovery by Algorithm 2}

As in the case of Algorithm 1,
we begin with a theorem that collects theoretical error guarantees. 

\begin{theorem}\label{case1_theorem}
    Let $u \in \HH$ be a mild solution of (\ref{model}) where $f(t)$ is given as in (\ref{Reducedf}). Suppose that $\sup_{j} \|h_j\| = H$, $\sup_{g\in \Tilde{G}} \|g\| = R$, $\inf_{j} \rho_j = \check{\rho}$, $\sup_{j} \rho_j = \hat{\rho}$, and $\beta < \frac{2\pi k}{\hat{\rho}}$, where $k\in\Z\setminus\{0\}$. Let $ \Tilde{\rho}_j,\esshape$ and $\estime_j = \ell\beta$ be the output of Algorithm 2. Then $|\estime_j - t_j| \leq \beta$, and
    \begin{align}
        |\esshape - \ang{h_j,g}| &\leq  H\|g\|\bigg(e^{\beta(3|\tilde{\rho}_j-\rho_j| + \hat{\rho})} - 1 + \frac{|\tilde{\rho}_j-\rho_j|\beta e^{\beta(3|\tilde{\rho}_j-\rho_j| + \hat{\rho})}}{\sqrt{\check{\rho}^2\beta^2 + 4\pi^2k^2}}\bigg) \label{coeff_error}\\
    &\quad+ 2\sqrt{2}\beta e^{3\hat{\rho}\beta}\left( \|g\|(\frac{12}{\pi}L_\ell + H\hat{\rho}) + 6\sigma\right) \nonumber\\
    &\leq \beta\bigg(H\|g\|\big((3(\hat{\rho}-\check{\rho})+\hat{\rho})e^{\beta(3(\hat{\rho}-\check{\rho})+\hat{\rho})} + \frac{(\hat{\rho} - \check{\rho})e^{\beta(3(\hat{\rho}-\check{\rho})) + \hat{\rho})}}{2\pi k}\big) \nonumber\\
    &\quad+ 2\sqrt{2} \beta e^{3\hat{\rho}\beta}\big( \|g\|(\frac{12}{\pi}L_\ell + H\hat{\rho}) + 6\sigma\big) \bigg)\nonumber.
    \end{align}
    Furthermore, let 
    \begin{align*}
        M(g) = \beta|\Delta_{s,\ell+1}(g,\beta) - \Delta_{s,\ell-2}(g,\beta)|
    \end{align*}
    where the $\Delta_{s,\ell}(g,\beta)$ are as in $(38)$. If $M(g) > \frac{12}{\pi}L_\ell \|g\| + 4\sigma$, then 
    \begin{align}
    |\tilde{\rho}_j - \rho_j| &\leq \frac{\frac{4}{\pi}L_\ell\|g\|(1 + 3\hat{\rho}\beta) + 4\sigma(1 + \hat{\rho}\beta)}{ M(g)} + \hat{\rho} - \frac{1 - e^{-\hat{\rho}\beta}}{\beta} \label{rate_error}\\
    &\leq \frac{4L_\ell\|g\|}{\pi M(g)} + 4\sigma\left(\frac{1+\hat{\rho}\beta}{ M(g)}\right) + \beta\left(\frac{\hat{\rho}^2}{2} + \frac{12L_\ell\|g\|\hat{\rho}}{\pi M(g)}\right)\nonumber
    \end{align}
\end{theorem}

\begin{remark}
    \emph{
    \begin{enumerate}
        \item The error in $(\ref{coeff_error})$ tends to 0 as $\beta \to 0$. Moreover, if $\tilde{\rho}_j = \rho_j$, then (\ref{coeff_error}) becomes
        \[
        |\esshape - \ang{h_j,g}| \leq H\|g\|(e^{\beta}-1) + \sqrt{2}\beta e^{3\hat{\rho}\beta}\left(\frac{24}{\pi}L_\ell\|g\| + 12\sigma\right).
        \]
        \item The estimates in (\ref{coeff_error}) can, in fact, be improved. In Section \ref{proof_section}, we give more precise bounds for the error in the coefficient by examining three different cases. 
        \item In the ideal scenario $L_\ell = \sigma = 0$, the error in the recovery of $\rho_j$ is bounded above as 
        \[
        |\tilde{\rho}_j - \rho_j| \leq \hat{\rho} - \frac{1 - e^{-\hat{\rho}\beta}}{\beta} \leq \frac{\hat{\rho}^2}{2}\beta,
        \]
        which clearly tends to 0 as $\beta \to 0$. 
        \item For each $g \in \mathcal{G}$, the value $M(g)$ stabilizes to a nonzero value as $\beta \to 0$. In fact, in the ideal scenario, 
        \[
        M(g) \to |\ang{h_j,g}|.
        \]
    \end{enumerate}
    }
\end{remark}

\subsection{Details of Algorithm 2: Prony-Laplace Method}
Recall that Algorithm 2 uses measurements given by 
\eqref{eq_abstract_m}.
These measurement values can be rewritten in a more practical form, as demonstrated by the following lemma, the proof of which can 
be found in the Appendix.

\begin{lemma}\label{derivations}
    Assume $A$ is the generator of a $C_0$-semigroup and $u_0\in\HH$. If $t_j \in [\ell\beta,(\ell+1)\beta]$ where $\ell \in \Z$ and $s = 2\pi i k/\beta$ for some $k \in \Z$, then 
    \begin{equation}\label{eq_m_in_int}
    \begin{split}
    m_{s,\ell}(g,\beta) &= \frac{e^{-st_j} - e^{\rho_j(t_j - (\ell + 1)\beta})}{(\rho_j + s)\beta^2}\ang{h_j,g} \\
    &+ \frac{1}{\beta^2}\int_{\ell\beta}^{(\ell+1)\beta} e^{-st}\ang{\tilde{\eta}(t),g} \; dt + \nu_{s,\ell}(g) + \ang{\varphi(\ell,\beta,t_j),\frac{g}{\beta^2}}. 
\end{split}
\end{equation} 
Similarly, if $t_j < \ell\beta$, then 
\begin{equation} 
\begin{split}
    m_{s,\ell}(g,\beta) &= \frac{e^{\rho_j(t_j - \ell\beta)} - e^{\rho_j(t_j - (\ell + 1)\beta})}{(\rho + s)\beta^2}\ang{h_j,g} \\
    &+ \frac{1}{\beta^2}\int_{\ell\beta}^{(\ell+1)\beta} e^{-st}\ang{\tilde{\eta}(t),g} \; dt + \nu_{s,\ell}(g) + \ang{\tilde{\varphi}(\ell,\beta,t_j),\frac{g}{\beta^2}}. 
\end{split}
\end{equation}
\end{lemma}

A particularly important feature of these quantities is that the values $\varphi$ and $\Tilde{\varphi}$ do not depend on the value of $k$ defining $s = 2\pi i k/\beta$.

Algorithm 2 makes use of the values $\Delta_{s,\ell}(g,\beta) = m_{s,\ell}(g,\beta) - m_{0,\ell}(g,\beta)$, which, in lieu of Lemma \ref{derivations}, are given by
\begin{equation}\label{eq_delta_in_int}
\begin{split}
\Delta_{s,\ell}(g,\beta) &= m_{s,\ell}(g, \beta) - m_{0,\ell}(g, \beta) \\
                &= \frac{\rho_j (e^{-st_j}-1) + s(e^{\rho_j(t_j - (\ell+1)\beta)} - 1)}{\rho_j(\rho_j+s)\beta^2}\ang{h_j,g} + \varepsilon_{s,\ell}(g,\beta)
\end{split}
\end{equation}
when $\ell\beta \leq t_j \leq (\ell + 1)\beta$, and 
\begin{equation}\label{eq_delta_not_in_int}
\begin{split}
   \Delta_{s,\ell}(g,\beta) &= m_{s,\ell}(g/\beta^2) - m_{0,\ell}(g/\beta^2) \\
   &=  s\frac{e^{\rho_j(t_j - (\ell+1)\beta)} - e^{\rho_j(t_j - \ell\beta)}}{\rho_j(\rho_j+s)\beta^2}\ang{h_j,g} + \varepsilon_{s,\ell}(g,\beta)
\end{split}
\end{equation}
when $t_j < \ell\beta$. The error term $\varepsilon_{s,\ell}(g,\beta)$ is given by 
\begin{equation}\
\varepsilon_{s,\ell}(g,\beta) = \frac{1}{\beta^2}\int_{\ell\beta}^{(\ell+1)\beta}(e^{-st}-1)\ang{\eta(t),g}\;dt + \nu_{s,\ell} - \nu_{0,\ell}.
\end{equation}

Assume, for the moment, that $\varepsilon_{s,\ell}(g,\beta) \equiv 0$ for all $s,\ell,\beta,$ and $g$, and furthermore suppose that $t_j \in [\ell\beta,(\ell+1)\beta)$. Then
\begin{align*}
   \lim_{\beta \to 0} \frac{\Delta_{s,\ell+1}(g,\beta) - \Delta_{s,\ell+2}(g,\beta)}{\beta\Delta_{s,\ell+1}(g,\beta)} &= \lim_{\beta \to 0} \frac{e^{\rho_j(t_j - (\ell+3)\beta)} - 2e^{\rho_j(t_j - (\ell+2)\beta)} + e^{\rho_j(t_j - (\ell+1)\beta)}}{\beta(e^{\rho_j(t_j - (\ell+1)\beta)}-e^{\rho_j(t_j - (\ell+2)\beta)})}\\
    &= \lim_{\beta \to 0} \frac{1 - e^{-\rho_j\beta}}{\beta} \\
    &= \rho_j.
\end{align*}
Thus, it is natural to approximate $\rho_j$ by 
\[
\tilde{\rho}_j := \frac{\Delta_{s,\ell+2}(g,\beta) - \Delta_{s,\ell+1}(g,\beta)}{\beta\Delta_{s,\ell+1}(g,\beta)}.
\]
Also, 
\begin{align*}
    \ang{h_j,g} &= \frac{\rho_j(\rho_j+s)\beta^2}{s(e^{\rho_j(t_j - (\ell+2)\beta)} - e^{\rho_j(t_j - (\ell+1)\beta)})}\Delta_{s,\ell+1}(g,\beta) \\ 
&\approx \frac{\tilde{\rho}_j(\tilde{\rho}_j+s)\beta^2}{s(e^{-2\tilde{\rho}_j\beta} - e^{-\tilde{\rho}_j\beta})}\Delta_{s,\ell+1}(g,\beta).
\end{align*}

This brief examination of the ideal scenario illuminates the core of Algorithm 2 (in this  case $\Delta_{s,\ell-2}(g,\beta)=0$). To account for noise, we introduce threshold values as in Algorithm 1. One of the thresholds will be used to determine if a catalyst entered the system in the interval $[\ell\beta,(\ell+1)\beta)$, and another one to cut off  the coefficients $\ang{h_j,g}$ of insufficiently large magnitude. The first threshold will, naturally, coincide with the upper bound on the magnitude of measurements when there is nothing to detect. In particular,  
\begin{equation}\label{error_estim}
\begin{split}
|\varepsilon_{s,\ell}(g,\beta) - \varepsilon_{s,\ell-n}(g,\beta)| 
&\leq \frac{1}{\beta^2}\bigg|\int_{\ell\beta}^{(\ell+1)\beta}(e^{-st}-1)\ang{\eta(t) - \eta(t - n\beta),g}\;dt \bigg| + 4\sigma \\
&\leq \frac{nL\|g\|}{\beta}\int_{\ell\beta}^{(\ell+1)\beta} |e^{-st}-1| \; dt + 4\sigma \\
&= \frac{4n}{\pi}L\|g\| + 4\sigma,
\end{split}
\end{equation}
and $|\Delta_{s,\ell}(g,\beta) - \Delta_{s,\ell - n}(g,\beta)| = |\varepsilon_{s,\ell}(g,\beta) - \varepsilon_{s,\ell-n}(g,\beta)|$ when $(\ell+1)\beta < t_j$. Thus, we define 
\begin{equation}\label{Thresh2}
    Q_n(g) = \frac{4n}{\pi}L\|g\| + 4\sigma, \ n\in\{1,2,3\}.
\end{equation}
The second threshold is given by 
\[
Q_*(g) = Q_1(g) + \hat{\rho}H\|g\|.
\]

To motivate this latter threshold value,  let us examine the measurements (\ref{eq_delta_not_in_int}). Let $\ell_0\beta$ be the initial time  passed to Algorithm 2, and suppose $t_j \in [\ell\beta,(\ell+1)\beta)$ for some integer $\ell > \ell_0$. Then, for all $n \geq 2$, as long as $(\ell+n-1)\beta < t_{j+1}$, we have 
\begin{align*}
|\Delta_{s,\ell+n}(g,\beta) - \Delta_{s,\ell+n-1}(g,\beta)| &= \bigg|s\frac{e^{\rho_j(t - (\ell+n+1)\beta)}(e^{\rho_j\beta}-1)^2}{\beta^2\rho_j(\rho_j+s)}\ang{h_j,g} + \varepsilon_{s,\ell}(g,\beta) - \varepsilon_{s,\ell-1}(g,\beta)\bigg| \\
&\leq \frac{2\pi ke^{-2\rho_j\beta}(e^{\rho_j\beta}-1)^2}{\rho_j\beta^2\sqrt{\rho_j^2\beta^2 + 4\pi^2k^2}}|\ang{h_j,g}| + |\varepsilon_{s,\ell}(g,\beta) - \varepsilon_{s,\ell-1}(g,\beta)| \\
&\leq\rho_j H\|g\| + Q_1(g) \\
&\leq Q_*(g).
\end{align*}
 This means that if $\ell' \in \Z$ is the smallest integer for which $\ell' > \ell_0$ and 
\[
\sup_{g\in \mathcal{G}} \big(|\Delta_{s,\ell'}(g,\beta) - \Delta_{s,\ell' - 1}(g,\beta)| - Q_*(g) \big) > 0,
\]
then $t_j \in [\ell'\beta,(\ell'+1)\beta)$ or $t_j \in [(\ell'-1)\beta,\ell'\beta)$. Either way, $|\ell'\beta - t_j| \leq \beta$. Since Algorithm 2 sets $\estime_j = \ell'\beta$, we have shown that $|\estime_j - t_j| < \beta$ when the intake is detected. 

In recovering the decay rate and coefficient, there are three cases to consider. The first is when $t_j \in [\ell\beta,(\ell+1) \beta)$ and 
\begin{align*}
    \max\{ &\sup_{g\in \mathcal{G}}\big(| \Delta_{s,\ell}(g,\beta) - \Delta_{s,\ell-1}(g,\beta)| - Q_*(g)\big), \\
    &\sup_{g\in \mathcal{G}}\big(|\Delta_{s,\ell+1}(g,\beta) - \Delta_{s,\ell}(g,\beta)| - Q_*(g)\big)\} \leq 0.
\end{align*}
When this happens, Algorithm 2 does not detect the catalyst at time $t_j$, and no data can be recovered. For this case, a uniform upper bound for $|\ang{h_j,g}|$ is given in Proposition \ref{no_recovery}. 

The second case occurs when there exists an $\ell > \ell_0$ such that 
\begin{equation}\label{recovery_condition}
\sup_{g\in \mathcal{G}}\big(| \Delta_{s,\ell}(g,\beta) - \Delta_{s,\ell-1}(g,\beta)| - Q_*(g)\big) > 0,
\end{equation}
and for a particular $g \in \mathcal{G}$, 
\[
|\Delta_{s,\ell+1}(g,\beta) - \Delta_{s,\ell-2}(g,\beta)| \leq Q_3(g) 
\]
(recall from the discussion above that this means $t_j \in [(\ell-1)\beta,(\ell+1)\beta))$. The measurements against this $g\in \mathcal{G}$ will not be used to recover the decay rate, and Algorithm 2 will set $\esshape := 0$. For this case, an upper bound on $|\esshape - \ang{h_j,g}| = |\ang{h_j,g}|$ is given in Proposition \ref{coeff_eq_zero}. 

The final case occurs when, again, there is an $\ell > \ell_0$ such that (\ref{recovery_condition}) holds and, for a particular $g \in \mathcal{G}$, 
\[
|\Delta_{s,\ell+1}(g,\beta) - \Delta_{s,\ell-2}(g,\beta)| > Q_3(g). 
\]
The measurements against such a $g \in \mathcal{G}$ can indeed be used to approximately recover the rate $\rho_j$, and the error of this approximation is bounded in Proposition \ref{rate_recovery}. Finally, an upper bound for the error $|\esshape - \ang{h_j,g}|$ in this case is given in Proposition \ref{full_recovery}.

As the above three cases cover all possible scenarios, Theorem \ref{case1_theorem} will be established once the above mentioned propositions are proved.

\subsection{Proof of Theorem \ref{case1_theorem}}\label{proof_section}
We examine the cases discussed above individually. 

\begin{proposition}\label{no_recovery}
    Assume the hypotheses of Theorem \ref{case1_theorem}. Suppose $t_j \in [\ell\beta,(\ell+1)\beta)$ and 
    \begin{align*}
    \max\{ &\sup_{g\in \mathcal{G}}\big(| \Delta_{s,\ell}(g,\beta) - \Delta_{s,\ell-1}(g,\beta)| - Q_*(g)\big), \\
    &\sup_{g\in \mathcal{G}}\big(|\Delta_{s,\ell+1}(g,\beta) - \Delta_{s,\ell}(g,\beta)| - Q_*(g)\big)\} \leq 0.
\end{align*}
Then 
\[
|\ang{h_j,g}| \leq 2\sqrt{2}\beta e^{2\hat{\rho}\beta}\left(\|g\|(\frac{8}{\pi}L + \hat{\rho}H) + 12\sigma\right)
\]
for all $g \in \mathcal{G}$. 
\end{proposition}

\begin{proof}
    Let $g \in \mathcal{G}$. Then, by assumption 
    \[
    |\Delta_{s,\ell}(g,\beta) - \Delta_{s,\ell-1}(g,\beta)| \leq Q_*(g)
    \]
    and
    \[
    |\Delta_{s,\ell+1}(g,\beta) - \Delta_{s,\ell}(g,\beta)| \leq Q_*(g).
    \]
    From this it follows that
    \begin{align*}
    |\Delta_{s,\ell+1}(g,\beta) - \Delta_{s,\ell-1}(g,\beta)| 
    &=|\Delta_{s,\ell+1}(g,\beta)-\Delta_{s,\ell}(g,\beta)+\Delta_{s,\ell}(g,\beta)-\Delta_{s,\ell-1}(g,\beta)|\\
    &\le|\Delta_{s,\ell+1}(g,\beta)-\Delta_{s,\ell}(g,\beta)|+|\Delta_{s,\ell}(g,\beta)-\Delta_{s,\ell-1}(g,\beta)|\\
    &\le 2Q_*(g).
    \end{align*}
    Now consider the string of inequalities: 
    \begin{align*}
        2Q_*(g) &\ge |\Delta_{s,\ell+1}(g,\beta) - \Delta_{s,\ell-1}(g,\beta)| \\
        &= \bigg|s\frac{e^{\rho_j(t_j - (\ell+2)\beta)} - e^{\rho_j(t_j - (\ell+1)\beta)}}{\rho_j(\rho_j+s)\beta^2}\ang{h_j,g} + \varepsilon_{s,\ell+1}(g,\beta) - \varepsilon_{s,\ell-1}(g,\beta)\bigg| \\
        &\geq \frac{2\pi k(e^{\rho_j(t_j - (\ell+1)\beta)} - e^{\rho_j(t_j - (\ell+2)\beta)})}{\rho_j\beta^2\sqrt{\rho_j^2\beta^2+4\pi^2k^2}}|\ang{h_j,g}| - Q_2(g) \\
        &\geq \frac{e^{-\rho_j\beta} - e^{-2\rho_j\beta}}{\sqrt{2}\rho_j\beta^2}|\ang{h_j,g}| - Q_2(g),
    \end{align*}
    where we have used the assumption that $\beta < 2\pi k/\hat{\rho}$ (and hence $\rho_j^2\beta^2 < 4\pi^2k^2$). An elementary calculation shows that 
    \[
    \frac{e^{-\rho_j\beta} - e^{-2\rho_j\beta}}{\sqrt{2}\rho_j\beta^2} \geq \frac{1}{\sqrt{2}\beta e^{2\rho_j\beta}}\geq \frac{1}{\sqrt{2}\beta e^{2\hat{\rho}\beta}}. 
    \]
    Therefore 
    \[
    2Q_*(g) > \frac{1}{\sqrt{2}\beta e^{2\hat{\rho}\beta}}|\ang{h_j,g}| - Q_2(g).
    \]
    Rearranging the above gives the stated result.
\end{proof}

\begin{proposition}\label{coeff_eq_zero}
    Assume the hypotheses of Theorem \ref{case1_theorem}. Further assume that there is an $\ell > \ell_0$ for which (\ref{recovery_condition}) holds, and that $\ell$ is the smallest integer satisfying this condition. Given $g\in \mathcal{G}$, if 
    \[
|\Delta_{s,\ell+1}(g,\beta) - \Delta_{s,\ell-2}(g,\beta)| \leq Q_3(g), 
\]
then $\esshape = 0$ and 
\[
|\esshape - \ang{h_j,g}| = |\ang{h_j,g}| < 8\sqrt{2}\beta e^{3\hat{\rho}\beta}(\frac{3}{\pi}L\|g\| + \sigma).
\]
\end{proposition}

The proof of the above proposition involves computations nearly identical to those in Proposition \ref{no_recovery}, so we omit it. We now only need to prove the estimates for the third case. We prove the bounds on the error in the rate of decay first. 

\begin{proposition}\label{rate_recovery}
     Assume the hypotheses of Theorem \ref{case1_theorem}. Further assume that there is an $\ell > \ell_0$ for which (\ref{recovery_condition}) holds, and that $\ell$ is the smallest integer satisfying this condition. Assume $g\in \mathcal{G}$ is such that 
     \[
    |\Delta_{s,\ell+1}(g,\beta) - \Delta_{s,\ell-2}(g,\beta)| > Q_3(g).
     \]
     Then 
     \[
     |\Tilde{\rho}_j - \rho_j| \leq \frac{\frac{4}{\pi}L_\ell\|g\|(1 + 3\hat{\rho}\beta) + 4\sigma(1 + \hat{\rho}\beta)}{ \beta|\Delta_{s,\ell+1}(g,\beta) - \Delta_{s,\ell-2}(g,\beta)|} + \hat{\rho} - \frac{1 - e^{-\hat{\rho}\beta}}{\beta}.
     \]
\end{proposition}
     \begin{proof}
         It is convenient to define the following values:
    \begin{align*}
        A &= -s\frac{e^{\rho_j(t_j-(\ell+3)\beta)}(e^{\rho_j\beta}-1)^2}{\rho_j(\rho_j+s)\beta^2} \\
        B &= \beta s\frac{e^{\rho_j(t_j - (\ell+2)\beta)}-e^{\rho_j(t_j - (\ell+1)\beta)}}{\rho_j(\rho_j+s)\beta^2}\\
        \delta_1 &=  \varepsilon_{s,\ell+1}(g,\beta) - \varepsilon_{s,\ell+2}(g,\beta) \\
        \delta_2 &= \beta(\varepsilon_{s,\ell+1}(g,\beta) - \varepsilon_{s,\ell-2}(g,\beta))
    \end{align*}
    This way, 
    \begin{equation*}
        \Tilde{\rho}_j = \frac{\Delta_{s, \ell + 1}(g,\beta) - \Delta_{s,\ell+2}(g,\beta)}{\beta(\Delta_{s,\ell+1}(g,\beta)-\Delta_{s,\ell-2}(g,\beta))}  = \frac{A+\delta_1}{B+\delta_2}, \text{ and } \frac{A}{B} = \frac{1 - e^{-\rho_j\beta}}{\beta}.
    \end{equation*}
    In this notation, we can estimate:
    \begin{align*}
        |\frac{A+\delta_1}{B+\delta_2} - \rho_j| &\leq |\frac{A+\delta_1}{B+\delta_2}-\frac{A}{B}| + |\frac{A}{B} - \rho_j| \\ 
        &= \bigg|\frac{B\delta_1 - A\delta_2}{B(B+\delta_2)}\bigg| + \bigg(\rho_j - \frac{1 - e^{-\rho_j\beta}}{\beta}\bigg) \\
        &\leq \frac{|\delta_1| + \frac{A}{B}|\delta_2|}{|B+\delta_2|} + \bigg(\rho_j - \frac{1 - e^{-\rho_j\beta}}{\beta}\bigg) \\
        &\leq \frac{|\delta_1| + \rho_j|\delta_2|}{|B+\delta_2|} + \bigg(\rho_j - \frac{1 - e^{-\rho_j\beta}}{\beta}\bigg).
    \end{align*}
    
Using (\ref{error_estim}), we have the upper bounds 
\begin{equation*}
    |\delta_1| \leq \frac{4}{\pi}L\|g\| + 4\sigma \text{ and } \rho_j|\delta_2| \leq \beta(\frac{12}{\pi}L\|g\|\hat{\rho} + 4\sigma\hat{\rho}),
\end{equation*}
so 
\[
|\delta_1| + \rho_j|\delta_2| \leq \frac{4}{\pi}L\|g\|(1 + 3\hat{\rho}\beta) + 4\sigma(1 + \hat{\rho}\beta).
\]
We conclude that 
\begin{align*}
    |\tilde{\rho}_j - \rho_j| &\leq \frac{\frac{4}{\pi}L_\ell\|g\|(1 + 3\hat{\rho}\beta) + 4\sigma(1 + \hat{\rho}\beta)}{\beta|\Delta_{s,\ell+1}(g,\beta)-\Delta_{s,\ell-2}(g,\beta)|} + \hat{\rho} - \frac{1 - e^{-\hat{\rho}\beta}}{\beta}
\end{align*}
as claimed.
     \end{proof}

It remains to estimate the error in recovering the coefficient in the case when 
\[
|\Delta_{s,\ell+1}(g,\beta) - \Delta_{s,\ell-2}(g,\beta)| > Q_3(g).
\]
This estimate is the most technical in this section and needs the following  lemma. 
\begin{lemma}\label{lemmaIHate}
Suppose $\beta > 0$, and consider the function 
\[
f(x) = \frac{e^{\beta x}-1}{x}, \quad f(0)=\beta.
\]
Then, for any $a \in \R$, the function 
\[
g_a(x) = \left|1 - \frac{f(x)}{f(x+a)}\right|
\]
is nondecreasing on $[0,\infty)$. In particular, 
\[
g_a(x) \leq \lim_{x \to \infty} g_a(x) = |1 - e^{-a\beta}| \leq e^{|a|\beta}-1
\]
\end{lemma}

The proof is given in the appendix. In the notation of the above lemma, it follows that if $a < 0$, then $f(x)/f(x+a) \geq 1$ for each $x \in [0,\infty)$, and that $f(x)/f(x+a)$ is increasing. This fact is used in (\ref{some_eq}) below. 

\begin{proposition}\label{full_recovery}
    Assume the hypotheses of Theorem \ref{case1_theorem}. Further assume that there is an $\ell > \ell_0$ for which (\ref{recovery_condition}) holds, and that $\ell$ is the smallest integer satisfying this condition. Assume $g\in \mathcal{G}$ is such that 
     \[
    |\Delta_{s,\ell+1}(g,\beta) - \Delta_{s,\ell-2}(g,\beta)| > Q_3(g).
     \]
     Then 
     \begin{align*}
         |\esshape - \ang{h_j,g}| &\leq  H\|g\|\bigg(e^{\beta(3|\Tilde{\rho}_j-\rho_j| + \hat{\rho})} - 1 + \frac{|\epsilon_\rho|\beta e^{\beta(3|\tilde{\rho}_j-\rho_j|) + \hat{\rho})}}{\sqrt{\rho_j^2\beta^2 + 4\pi^2k^2}}\bigg)\\
        &+ 4\sqrt{2}\beta e^{2\hat{\rho}\beta}\left(\frac{3}{\pi}L\|g\| + \sigma\right).
     \end{align*}
\end{proposition}
\begin{proof}
    Suppose that 
    \[
|\Delta_{s,\ell+1}(g,\beta) - \Delta_{s,\ell-2}(g,\beta)| > Q_3(g),
\]
and define $\epsilon_\rho = \tilde{\rho}_j - \rho_j$. Furthermore, we define the function: 
\[
    \theta(\rho,t,\beta,\ell) = \frac{s(e^{\rho(t - (\ell + 2)\beta)} - e^{\rho(t - (\ell + 1)\beta)})}{\rho(\rho+s)\beta^2}. 
    \]
Algorithm 2 will define 
\begin{align*}
\esshape &:= \frac{\Tilde{\rho}_j(\tilde{\rho}_j+s)\beta^2}{s\cdot(e^{-2\Tilde{\rho}_j\beta} - e^{-\Tilde{\rho}_j\beta})} (\Delta_{s,\ell+1}(g,\beta)-\Delta_{s,\ell-2}(g,\beta))\\
&=\frac{\theta(\rho_j,t_j,\beta,\ell)}{\theta(\tilde{\rho}_j,\ell\beta,\beta,\ell)}\ang{h_j,g} + \frac{\varepsilon_{s,\ell+1}(g,\beta) - \varepsilon_{s,\ell-2}(g,\beta)}{\theta(\tilde{\rho}_j,\ell\beta,\beta,\ell)}.
\end{align*}
Using the triangle inequality, it follows that 
    \begin{align*}
        |\esshape - \ang{h_j,g}| &\leq \bigg| \frac{\theta(\rho_j,t_j,\beta,\ell)}{\theta(\tilde{\rho}_j,\ell\beta,\beta,\ell)} - 1 \bigg|\cdot |\ang{h_j,g}| + \bigg| \frac{\varepsilon_{s,\ell+1}(g,\beta) - \varepsilon_{s,\ell-2}(g,\beta)}{\theta(\tilde{\rho}_j,\ell\beta,\beta,\ell)}\bigg| \\
        &\leq \bigg| \frac{\theta(\rho_j,t_j,\beta,\ell)}{\theta(\tilde{\rho}_j,\ell\beta,\beta,\ell)} - 1 \bigg|H\|g\| + \bigg| \frac{\varepsilon_{s,\ell+1}(g,\beta) - \varepsilon_{s,\ell-2}(g,\beta)}{\theta(\tilde{\rho}_j,\ell\beta,\beta,\ell)}\bigg|.
    \end{align*}
    The latter term can be easily bounded: 
     \begin{align*}
        \bigg| \frac{\varepsilon_{s,\ell+1}(g,\beta) - \varepsilon_{s,\ell-2}}{\theta(\tilde{\rho}_j,\ell\beta,\beta,\ell)}\bigg| &= \bigg|\frac{\tilde{\rho}_j (\tilde{\rho}_j + s)\beta^2}{s(e^{-2\tilde{\rho}_j\beta}-e^{-\tilde{\rho}_j\beta})}\bigg|\cdot (\frac{12}{\pi}L_\ell\|g\| + 4\sigma) \\
        &\leq \frac{\tilde{\rho}_j|\tilde{\rho}_j+s|\beta^2}{|s|(e^{-\tilde{\rho}_j\beta}-e^{-2\tilde{\rho}_j\beta})}(\frac{12}{\pi}L_\ell\|g\| + 4\sigma)\\
        &\leq \sqrt{2}\beta e^{2\hat{\rho}\beta}(\frac{12}{\pi}L_\ell\|g\| + 4\sigma),
    \end{align*}
    where we have used a similar calculation to that used in the proof of Proposition \ref{no_recovery}. 

    As for the other term, using $s=2\pi  i k/\beta$, we have
    \begin{align*}
        \bigg|\frac{\theta(\rho_j,t_j,\beta,\ell)}{\theta(\tilde{\rho}_j,\ell\beta,\beta,\ell)} - 1 \bigg| &= \bigg| e^{2\beta\epsilon_\rho + \rho_j(t_j - \ell\beta)} \frac{(e^{\rho_j\beta} - 1)(\rho_j + \epsilon_\rho)(\rho_j +  \epsilon_\rho + s)}{(e^{(\rho_j + \epsilon_\rho)\beta}-1)\rho_j(\rho_j + s)} - 1\bigg| \\
        &= \bigg| e^{2\beta\epsilon_\rho + \rho_j(t_j - \ell\beta)} \frac{(e^{\rho_j\beta} - 1)(\rho_j + \epsilon_\rho)(\rho_j\beta + 2\pi i k + \epsilon_\rho\beta)}{(e^{(\rho_j + \epsilon_\rho)\beta}-1)\rho_j(\rho_j\beta + 2\pi i k)} - 1\bigg| \\
        &\leq |1 - e^{2\beta\epsilon_\rho + \rho_j(t_j - \ell\beta)}| + e^{2\beta\epsilon_\rho+\rho_j(t_j-\ell\beta)} \bigg|\frac{(e^{\rho_j\beta} - 1)(\rho_j + \epsilon_\rho)(\rho_j\beta + 2\pi i k + \epsilon_\rho\beta)}{(e^{(\rho_j + \epsilon_\rho)\beta}-1)\rho_j(\rho_j\beta + 2\pi i k)} - 1\bigg| \\
        &\leq (e^{\beta(2|\epsilon_\rho| + \rho_j)}-1) + e^{\beta(2|\epsilon_\rho| + \rho_j)}\bigg[\big| 1 - \frac{(e^{\rho_j\beta} - 1)(\rho_j+\epsilon_\rho)}{(e^{(\rho_j+\epsilon_\rho)\beta}-1)\rho_j}\big| \\
        &+ \big|\frac{\epsilon_\rho\beta(e^{\rho_j\beta}-1)(\rho_j+\epsilon_\rho)}{(e^{(\rho_j+\epsilon_\rho)\beta}-1)\rho_j(\rho_j\beta + 2\pi i k)} \big|\bigg].
    \end{align*}
    By Lemma \ref{lemmaIHate},
    \[
    \big| 1 - \frac{(e^{\rho_j\beta} - 1)(\rho_j+\epsilon_\rho)}{(e^{(\rho_j+\epsilon_\rho)\beta}-1)\rho_j}\big| \leq e^{|\epsilon_\rho|\beta}-1.
    \]
    Similarly, 
    \begin{equation}\label{some_eq}
    \begin{split}
\big|\frac{\epsilon_\rho\beta(e^{\rho_j\beta}-1)(\rho_j+\epsilon_\rho)}{(e^{(\rho_j+\epsilon_\rho)\beta}-1)\rho_j(\rho_j\beta + 2\pi i k)} \big| &\leq \frac{|\epsilon_\rho|\beta}{\sqrt{\rho_j^2\beta^2 + 4\pi^2k^2}}\bigg(\lim_{\rho_j\to\infty} \frac{(e^{\rho_j\beta}-1)(\rho_j-|\epsilon_\rho|)}{(e^{(\rho_j-|\epsilon_\rho|)\beta}-1)\rho_j} \bigg) \\
&= \frac{|\epsilon_\rho|\beta e^{|\epsilon_\rho|\beta}}{\sqrt{\rho_j^2\beta^2 + 4\pi^2k^2}}.
\end{split}
\end{equation}
Putting this all together yields 
    \begin{align*}
        |\esshape - \ang{h_j,g}| &\leq H\|g\|\bigg(e^{\beta(2|\epsilon_\rho| + \rho_j)}-1 + e^{\beta(2|\epsilon_\rho| + \rho_j)}\bigg[e^{|\epsilon_\rho|\beta}-1 + \frac{|\epsilon_\rho|\beta e^{|\epsilon_\rho|\beta}}{\sqrt{\rho_j^2\beta^2 + 4\pi^2k^2}}\bigg]\bigg) \\
        &+ \sqrt{2}\beta e^{2\hat{\rho}\beta}(\frac{12}{\pi}L_\ell\|g\| + 4\sigma) \\
        &\leq H\|g\|\bigg(e^{\beta(3|\epsilon_\rho| + \hat{\rho})} - 1 + \frac{|\epsilon_\rho|\beta e^{\beta(3|\epsilon_\rho|) + \hat{\rho})}}{\sqrt{\rho_j^2\beta^2 + 4\pi^2k^2}}\bigg)\\
        &+ \sqrt{2}\beta e^{2\hat{\rho}\beta}(\frac{12}{\pi}L_\ell\|g\| + 4\sigma)
    \end{align*}
    as claimed.
\end{proof}
    
Notice, all upper bounds given for $|\esshape - \ang{h_j,g}|$ in Propositions \ref{no_recovery}, \ref{coeff_eq_zero}, and \ref{full_recovery}, are no greater than the first bound stated in Theorem \ref{case1_theorem}. To obtain the second inequality, we can replace $|\tilde{\rho}_j - \rho_j|$ with $\hat{\rho} - \check{\rho}$, and note $2\pi k \leq \sqrt{\hat{\rho}\beta^2 + 4\pi^2k^2}$ to get 
    \begin{align*}
        |\esshape - \ang{h_j,g}| &\leq H\|g\|\bigg(e^{\beta(3(\hat{\rho}-\check{\rho}) + \hat{\rho})} - 1 + \frac{(\hat{\rho}-\check{\rho})\beta e^{\beta(3(\hat{\rho}-\check{\rho}) + \hat{\rho})}}{2\pi k}\bigg)\\
        &+ \sqrt{2}\beta e^{3\hat{\rho}\beta}(\frac{24}{\pi}L_\ell \|g\| + 12\sigma).
    \end{align*}
    Using the formula for the remainder of the Taylor series, we can estimate:
    \[
    e^{\beta(3(\hat{\rho}-\check{\rho})+\hat{\rho})}-1 \leq \beta(3(\hat{\rho}-\check{\rho})+\hat{\rho})e^{\beta(3(\hat{\rho}-\check{\rho})+\hat{\rho})}.
    \]
    This results in the weaker upper bound 
    \begin{align*}
       |\esshape - \ang{h_j,g}| &\leq \beta\bigg(H\|g\|\big((3(\hat{\rho}-\check{\rho})+\hat{\rho})e^{\beta(3(\hat{\rho}-\check{\rho})+\hat{\rho})} + \frac{(\hat{\rho} - \check{\rho})e^{\beta(3(\hat{\rho}-\check{\rho})) + \hat{\rho})}}{2\pi k}\big) \nonumber\\
    &+ \sqrt{2}\beta e^{3\hat{\rho}\beta}(\frac{24}{\pi}L_\ell \|g\| + 12\sigma) \bigg).
    \end{align*}
    Moreover, the second inequality in (\ref{rate_error}) is obtained simply by using the Taylor series 
    \[
    \hat{\rho} - \frac{1-e^{-\hat{\rho}\beta}}{\beta} \leq \frac{\hat{\rho}^2}{2}\beta
    \]
    and rearranging. This completes the proof of the Theorem.

\section{Numerical Experiments}\label{sec5}
To assess the performance of our algorithms, we apply them to a specific initial value problem in $\HH = L^2([0,1])$. We use
\begin{equation*}
\begin{cases}
\dot{u}(t)=u(t)+\sum\limits_{i=1}^3 h_i e^{\rho_i(t_i-t)}\chi_{[t_i, \infty)}(t)+\eta(t)\\
u(0)=0\\
\end{cases}
\end{equation*}
with $h_1(x)=3\sin(x)$, $h_2(x)=2.5\cos(x)$, $h_3(x)=x+2$, $x\in[0,1]$, $t_1=0.25$, $t_2=2.54$, $t_3=4.78$, $t\in[0,5]$, $\rho_1= 1$, $\rho_2= 2$, $\rho_3= 3.$ We consider two different types of background sources: $(\eta(t))(x)=xe^{-Lt}$ and $(\eta(t))(x)=x\sin(Lt)$.
As the sensor functions,
we use $g_1(x)=1$, $g_2(x)=x$, and $g_3(x)=x^2$. We first provide a simulation in which $D=2$, $L=10^{-2}$, and the noise level $\sigma=10^{-3}$. We show the error in both algorithms in recovering $(\rho_j,t_j,\ang{h_i,g_j})$ ($i,j = 1,2,3$) with these parameters. Then, we illustrate the sensitivity of the (average) error in the recovery of $\ang{h_i,g_j}$ to change of the parameters $\beta,L$ and $\sigma$.

\subsection{Simulation} As noted above, the parameters $D = 2$, $L = 10^{-2}$, and $\sigma = 10^{-3}$ are fixed in this section until otherwise specified. Figure \ref{fig1} below shows the error in recovering the times and coefficients for Algorithm 1. 

\begin{figure}[H]
    \centering
    \subfloat[$\eta=xe^{-Lt}$]{\includegraphics[width=0.47\linewidth]{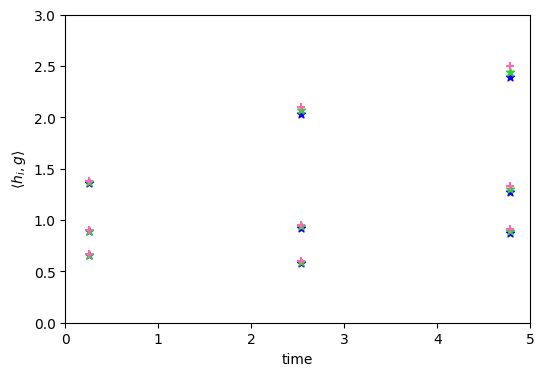}}
    \qquad
    \subfloat[$\eta=xsin(Lt)$]{\includegraphics[width=0.47\linewidth]{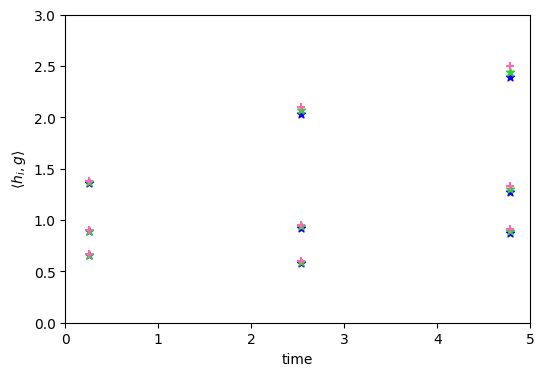}}
\caption{Simulation results for Algorithm 1. The results for $h_i$ lie in the $i$-th column. Pink plus signs stand for the ground truth $\langle h_i, g_j\rangle$, blue stars stand for the output $\mathfrak f_i(g_j)$ when $\beta=0.01$ and green stars stand for the output $\mathfrak f_i(g_j)$ when $\beta=0.005.$}
    \label{fig1}
\end{figure}

\noindent Figure $\ref{fig1}$ displays both the estimates and the ground truth in a single plot. The results indicate that our algorithms can successfully find all bursts, and the error decreases as the time step $\beta$ is reduced. Figure \ref{fig2} below provides a similar plot for Algorithm 2 using only the case when $\beta = .01$ (we omit the case when $\beta = 0.005$ as the results are not visually distinguishable from the case when $\beta = .01$).

\begin{figure}[H]
    \centering
    \subfloat[$\eta=xe^{-Lt}$]{\includegraphics[width=0.47\linewidth]{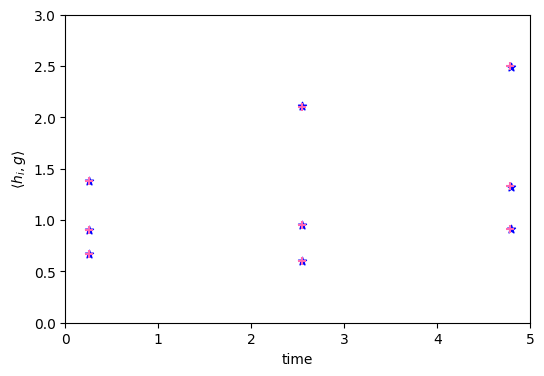}}
    \qquad
    \subfloat[$\eta=xsin(Lt)$]{\includegraphics[width=0.47\linewidth]{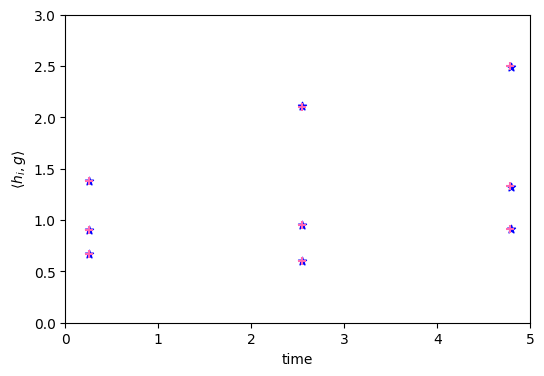}}
\caption{Simulation results for Algorithm 2. The results for $h_i$ lie in the $i$-th column. Pink plus signs stand for the ground truth $\langle h_i, g_j\rangle$, blue stars stand for the output $\mathfrak f_i(g_j)$ when $\beta=0.01$.}
    \label{fig2}
\end{figure}

We evaluated the accuracy of the estimates of the decay rate $\rho_j\ (j=1,2,3)$ by calculating the relative error:
$$\frac{|\rho_j-\Tilde{\rho}_j|}{|\rho_j|},\quad j=1,2,3.$$ Figure \ref{fig3} below presents the relative error in recovery of each $\rho_j$ for Algorithm 1 for various values of $N$ (defined in Theorem \ref{thm1}). We also show these errors for the ideal case when $L = \sigma = 0$. 

\begin{figure}[H]
    \centering
\subfloat[$\eta(x,t)=x\sin(Lt),\ L=10^{-2},\ \sigma=10^{-3},\ \beta=0.01$]
    {\includegraphics[width=0.47\linewidth]{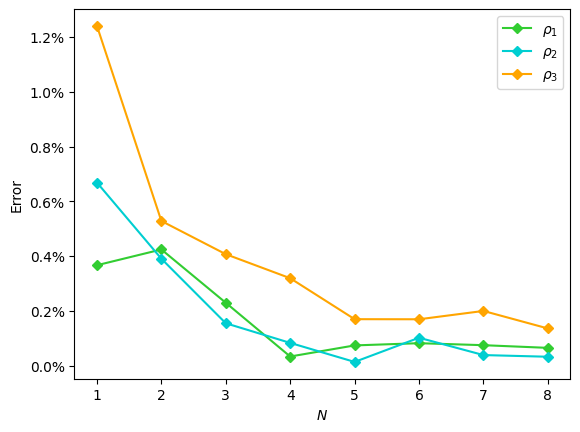}}
    \qquad 
    \subfloat[$L=\sigma=0,\ \beta=0.01$]
    {\includegraphics[width=0.47\linewidth]{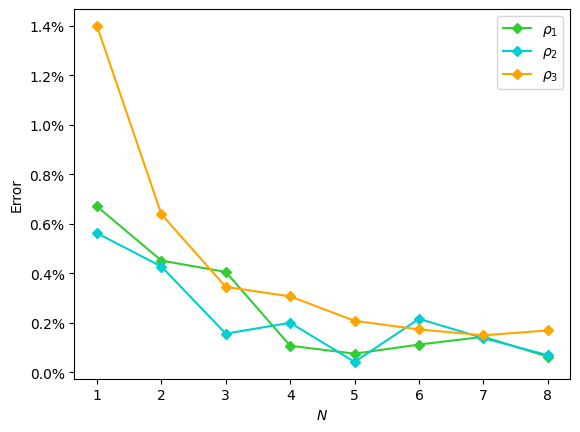}}
\caption{The error estimates of decay rate $\rho_j\ vs.\ N$ for Algorithm 1 in the simulation and the ideal model.}
    \label{fig3}
\end{figure}

\noindent Using Algorithm 2 in the simulation, the respective relative errors in the decay rates $\rho_1$, $\rho_2$, $\rho_3$ were approximately 1.01\%, 4.87\%, and 2.30\%. For the ideal case, the respective errors are approximately 0.50\%, 2.57\%, and 1.74\%. 

A few comments are in order before moving to the next section. First, the experiments above suggest that Algorithm 1 performs better in recovering the rate of decay, and Algorithm 2 is better for recovering the coefficients $\ang{h_i,g_j}$. This is consistent with the error bounds given in Theorems \ref{thm1} and \ref{case1_theorem}, and we can give an informal explanation for this trade-off. The methods of recovering the decay rate are similar between the two Algorithms, but the finer subdivision of the interval $[n\beta,(n+1)\beta)$ used in Algorithm 1 yields a more accurate estimate. In recovering the coefficient, both values $\mathfrak{m}_{i+1} - \mathfrak{m}_{i-2}$ and $\Delta_{s,\ell+1}(g,\beta) - \Delta_{s,\ell-2}(g,\beta)$, appearing respectively in Algorithms 1 and 2, have the form $c\cdot \ang{h_i,g} + \text{error}$ (although, of course, the values of $c$ and the error are not the same between the two). The approach of Algorithm 1 is simply to use this value to estimate $\ang{h_i,g}$ since, as $\beta \to 0$, $c \to 1$ and the only term contributing to the error is $\sigma$. In Algorithm 2, we use the estimated value of $\rho$ to approximate $1/c$, and multiplying $c\cdot \ang{h_i,g} + \text{error}$ through by this value produces a better approximation of the coefficient. 

\subsection{Varying Parameters}
Here we document three tests to show the sensitivity of the error in our approximations as the parameters $\beta, L,$ and $\sigma$ change. Specifically, we assessed the accuracy of the estimates for $\langle h_i, g_2\rangle$ by computing the relative error:
$$\frac{\sqrt{\sum\limits_{i=1}^3\left|\langle h_i, g_2\rangle-\mathfrak f_i(g_2)\right|^2}}{\sqrt{\sum\limits_{i=1}^3\left|\langle h_i, g_2\rangle\right|^2}}.$$
Note that due to the observations made in the last section, we expect the errors for Algorithm 2 to be smaller than those for Algorithm 1. 
\begin{figure}[H]
    \centering
    {\includegraphics[width=0.45\linewidth]{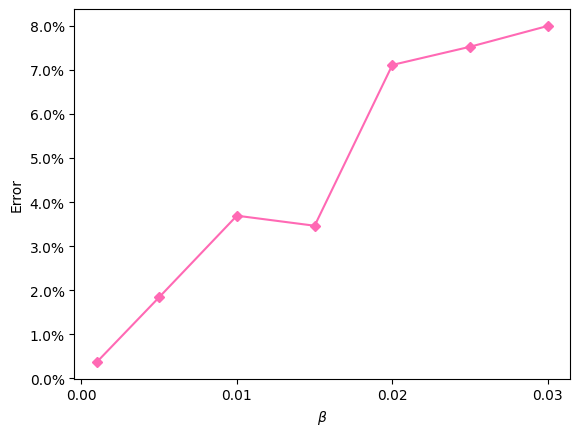}}
    \qquad
    {\includegraphics[width=0.45\linewidth]{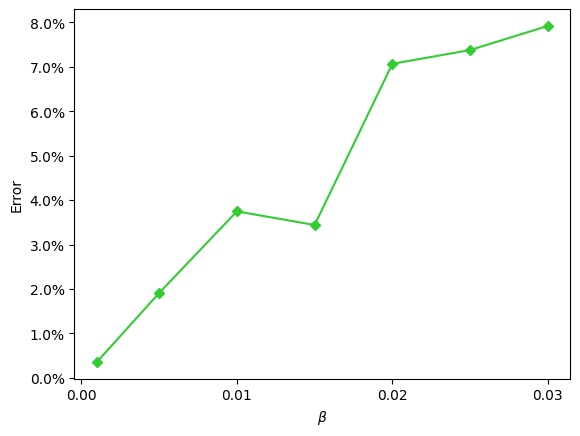}} 

    {\includegraphics[width=0.45\linewidth]{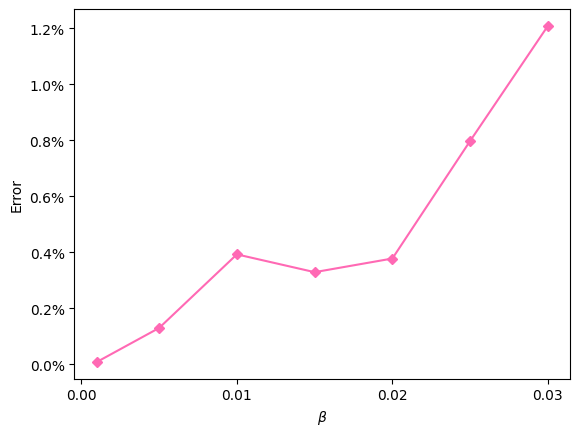}}
    \qquad
    {\includegraphics[width=0.45\linewidth]{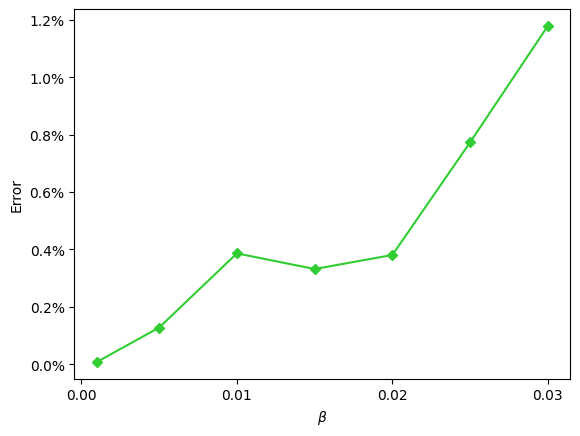}}
    
\caption{The error estimate  of $\langle h_i,g_2\rangle\ vs.\ \beta$: $L=10^{-2},$ $\sigma=10^{-3}$. The background sources are $\eta(x,t)=xe^{-Lt}$ and $\eta(x,t)=x\sin(Lt)$ for the first and second columns, respectively. Algorithm 1 is given in the top row, Algorithm 2 on the bottom.}
    \label{fig4}
\end{figure}

The results in Figure \ref{fig4} here are as expected. The relative error is quite low for both Algorithms for small values of $\beta$, but grow quickly as $\beta$ gets larger. The next group of plots, given in Figure \ref{fig5}, shows the error in both Algorithms as the Lipschitz constant $L$ is varied. Surprisingly, the performance of each Algorithm is almost independent of the Lipschitz constant in the range investigated.

\begin{figure}[H]
    \centering
    {\includegraphics[width=0.45\linewidth]{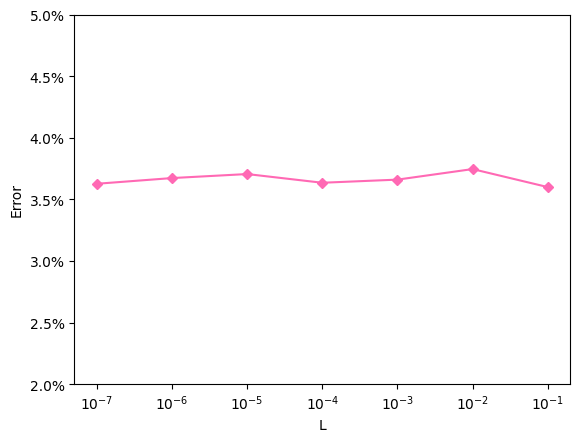}}
    \qquad
    {\includegraphics[width=0.45\linewidth]{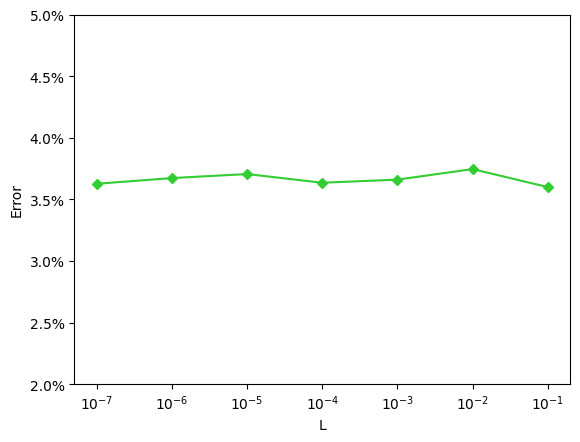}} 

    {\includegraphics[width=0.45\linewidth]{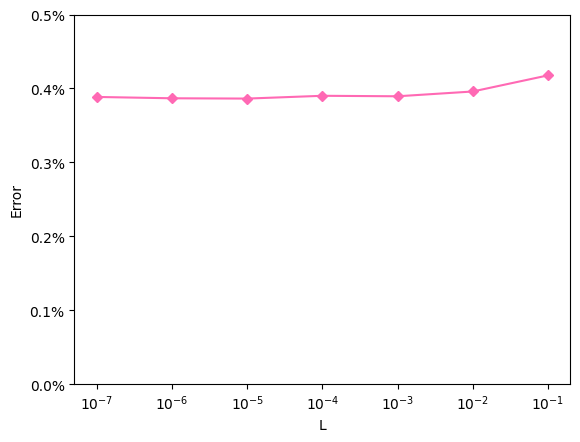}}
    \qquad
    {\includegraphics[width=0.45\linewidth]{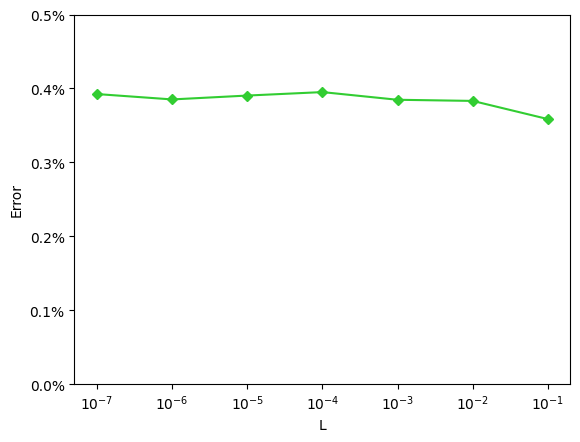}}
    
\caption{The error estimate  of $\langle h_i,g_2\rangle\ vs.\ L$: $\beta = 0.01,$ $\sigma=10^{-3}$. The background sources are $\eta(x,t)=xe^{-Lt}$ and $\eta(x,t)=x\sin(Lt)$ for the first and second columns, respectively. Algorithm 1 is given in the top row, Algorithm 2 on the bottom.}
    \label{fig5}
\end{figure}

The last numerical experiment, given in Figure \ref{fig6}, plots the accuracy of each Algorithm against the noise level $\sigma$. We observe that, for both Algorithms, the noise has minimal influence on the error when it is less $10^{-3}$. However, both Algorithms see a sharp increase in the error as $\sigma$ approaches $10^{-1}$. This is because when the additive noise $\sigma = 10^{-1}$, it accounts for roughly 10\% of the signal values. As a result, the error is primarily determined by the noise rather than the time-step $\beta$ or the Lipschitz constant $L$ of the background source. 

\begin{figure}[H]
    \centering
    {\includegraphics[width=0.45\linewidth]{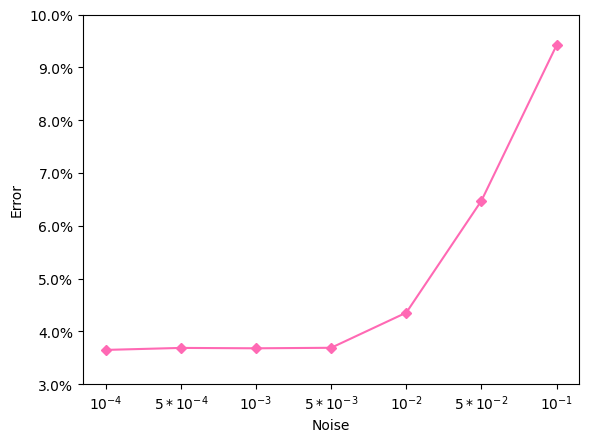}}
    \qquad
    {\includegraphics[width=0.45\linewidth]{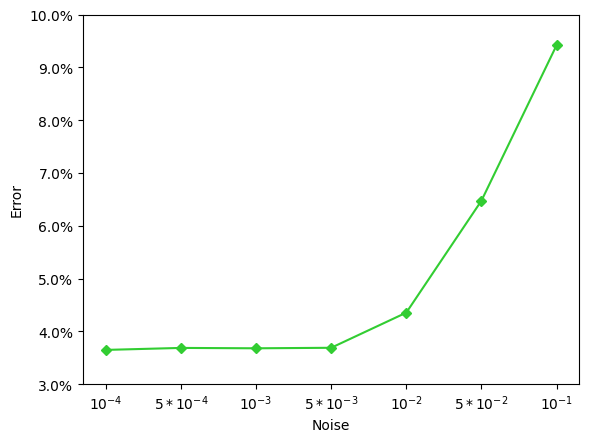}} 

    {\includegraphics[width=0.45\linewidth]{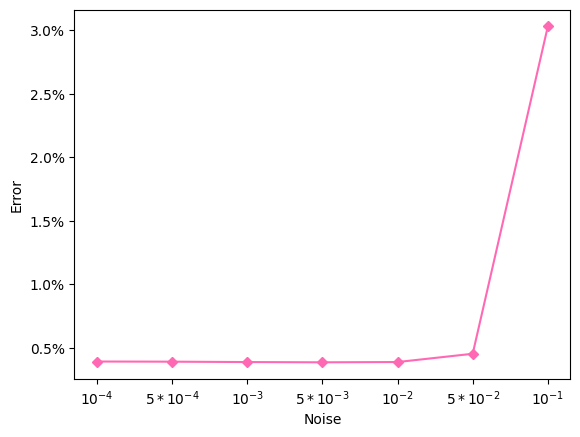}}
    \qquad
    {\includegraphics[width=0.45\linewidth]{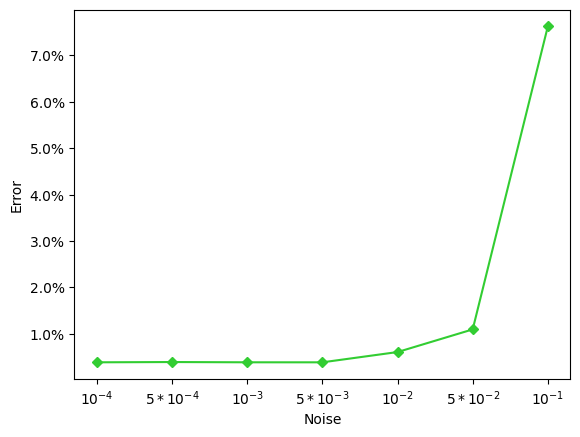}}
    
\caption{The error estimate  of $\langle h_i,g_2\rangle\ vs.\ \sigma$: $\beta = 0.01,$ $L=10^{-2}$. The background sources are $\eta(x,t)=xe^{-Lt}$ and $\eta(x,t)=x\sin(Lt)$ for the first and second columns, respectively. Algorithm 1 is given in the top row, Algorithm 2 on the bottom.}
    \label{fig6}
\end{figure}

\section{Appendix}
\subsection{Derivation of Local Lipschitz Constants} Here we derive local Lipschitz constants of $\Tilde{\eta}$ defined in (\ref{Reducedf}). Let $\ell_0\beta$ be the initial time-step passed to Algorithm 2 so that $t_{j-1} < \ell_0\beta < t_j$. For each time step $\ell\beta > \ell_0\beta$ around which the catalyst at time $t_j$ is not detected, we can update the Lipschitz constant of $\Tilde{\eta}$. Observe the string of inequalities
\begin{align*}
     \sup_{t \geq \ell\beta} \|\frac{d}{dt} (\sum_{i = 1}^{j-1} h_ie^{\rho_i(t_i - t)})\| &\leq 
 \sup_{t\geq \ell\beta}\bigg( \sum_{i = 1}^{j-1} \|h_i\|\rho_i e^{\rho_i(t_i - t)} \bigg) \\ 
 &\leq H\hat{\rho}\sum_{i = 0}^{j-2} e^{\check{\rho}(t_{j-1-i} - \ell\beta)}.
\end{align*}
Using the assumption \eqref{tsep}, and the fact that $t_{j-1} < \ell_0\beta$ we have that $t_{j-1-i} \leq \ell_0\beta - i(4\beta + D)$ for all $0 \leq i \leq j-2$. This implies that 
\begin{align*}
    H\hat{\rho}\sum_{i = 0}^{j-2} e^{\check{\rho}(t_{j-1-i} - \ell\beta)} &\leq H\hat{\rho}\sum_{i = 0}^{j-2} e^{\check{\rho}(\ell_0\beta - i(4\beta+D) - \ell\beta)} \\
    &\leq H\hat{\rho}e^{\beta(\ell_0-\ell)} \sum_{i = 0}^{\infty} e^{-i\check{\rho}(4\beta + D)} \\
    &= \frac{H\hat{\rho}e^{\beta(\ell_0 - \ell)}}{1 - e^{-\check{\rho}(4\beta + D)}}.
\end{align*}
It follows that $\Tilde{\eta}$ is Lipschitz on $[\ell\beta,\infty)$ with (known) Lipschitz constant 
\[
L_\ell = L + \frac{H\hat{\rho}e^{\beta(\ell_0 - \ell)}}{1 - e^{-\check{\rho}(4\beta + D)}} \le  L + \frac{H\hat{\rho}}{1 - e^{-\check{\rho} D}} .
\]

\subsection{Proof of Lemma \ref{lemmaIHate}} 
We show only the case when $a \geq 0$, the case for $a < 0$ is similar. Since $a \geq 0$, $f$ is increasing, so
    \[
    1 - \frac{f(x)}{f(x+a)} \geq 0
    \]
    for each $x \in [0,\infty)$. This means we only need to show that $f(x)/f(x+a)$ is monotonically nonincreasing. 
    By the quotient rule, it suffices to show that 
    \[
     f(x)f'(x+a) - f'(x)f(x+a) \geq 0
    \]
    for each $x$. An elementary calculation shows that 
    \begin{equation}\label{int_eq}
    \begin{split}
    &x^2(x+a)^2\big(f(x)f'(x+a) - f'(x)f(x+a)\big) \\
    &= a \big(1 + e^{\beta (2x + a)} + e^{\beta x} (\beta x -1) - e^{\beta (a + x)} (1 + \beta x)\big) - x^2\beta e^{\beta x}(e^{a\beta}-1).
    \end{split}
    \end{equation}
    Since the RHS of (\ref{int_eq}) vanishes for $x = 0$, it suffices to show that this RHS is nondecreasing. Taking the $x$-derivative of the RHS of (\ref{int_eq}) and simplifying gives 
    \begin{equation}\label{int_eq2}
    \beta e^{\beta x} ((e^{a \beta}-1) (2 + \beta x)x + 
   a (\beta x + 2 e^{\beta (a + x)}  - e^{a \beta} (2 + \beta x))).
    \end{equation}
    The first term in (\ref{int_eq2}) is nonnegative, and the second term is nonnegative for $x \in [0,\infty)$ if 
    \begin{equation}\label{int_eq3}
    2(e^{\beta(a+x)} - e^{a\beta}) + \beta x - e^{a\beta}\beta x \geq 0.
    \end{equation}
    This is, again, an expression which vanishes when $x = 0$, and is easily seen to be increasing. To summarize, (\ref{int_eq3}) is nonnegative, which implies (\ref{int_eq2}) is nonnegative, which means the RHS of (\ref{int_eq}) is nonnegative, so $f(x)/f(x+a)$ is nonincreasing. This completes the proof.

\subsection{Proof of Lemma \ref{derivations}} Let $T$
be the semigroup generated by the operator $A$. Then the measurements $(\ref{eq_abstract_m})$ can be  expanded as:

\begin{equation}\label{eq_abstract_m_expanded}
\begin{split}
m_{s,\ell}(g,\beta) &= \int_{\ell\beta}^{(\ell + 1)\beta} e^{-st} \bigg(T(t)u_0 + \chi_{[t_j,\infty)}\int_{t_j}^t e^{\rho_j(t_j - \tau)} \ang{T(t-\tau)h_j, (\overline{s}I-A^*)g}  \; d\tau \\
&\quad+ \int_{0}^t \ang{T(t-\tau)\eta(\tau),(\overline{s}I-A^*)g} \; d\tau\bigg) \;dt + \nu_{s,\ell}(g).
\end{split}
\end{equation}
The following Lemma was proven in \cite{aldhuakor23}, but we include the proof here for completeness. 

\begin{lemma}
    Let $\omega \in \R$ be the growth bound of the semigroup $T$,
i.e. the infimum over all real numbers such that there exists an $M \geq 0$ with $\|T(t)\| \leq Me^{\omega t}$ for all $t \geq 0$. Then, for all real numbers $s > \omega$, the following identity holds:
    \begin{equation}\label{eq_background}
    \begin{split}
    &\int_{\ell\beta}^{(\ell+1)\beta} e^{-s t} \ang{\int_{0}^t \ang{T(t-\tau)\eta(\tau),(\overline{s}I-A^*)\frac{g}{\beta^2}} \; d\tau} \; dt \\
    =& \int_{\ell\beta}^{(\ell+1)\beta} e^{-s t}\ang{\eta(\tau),\frac{g}{\beta^2}} \; dt + \int_0^{\ell\beta}e^{-s\ell\beta}\ang{T(\ell\beta - \tau)\eta(\tau),\frac{g}{\beta^2}} \\
    &\quad-\int_0^{(\ell+1)\beta}e^{-s(\ell+1)\beta}\ang{T((\ell+1)\beta - \tau)\eta(\tau),\frac{g}{\beta^2}}.
    \end{split}
\end{equation}
\end{lemma}

\begin{proof} Let $p \in \{\ell,\ell+1\}$. Then
\begin{align*} 
&\int_{p\beta}^{\infty} e^{-s t}\left\langle \int_{0}^{t}T(t-\tau)\eta(\tau)d\tau, \left(\overline
{s} I- A^*\right)\frac{g}{\beta^2}\right\rangle dt \\
=&\int_{0}^{p\beta}  \int_{p\beta}^{\infty}  e^{-s t}\left\langle T(t-\tau)\eta(\tau), \left(\overline
{s} I- A^*\right)\frac{g}{\beta^2}\right\rangle dtd\tau\\
&\quad+\int_{p\beta}^{\infty}  \int_{\tau}^{\infty}  e^{-s t}\left\langle T(t-\tau)\eta(\tau), \left(\overline
{s} I- A^*\right)\frac{g}{\beta^2}\right\rangle dtd\tau\\
=&\int_{0}^{p\beta}  \int_{0}^{\infty}  e^{-s(t+p\beta) }\left\langle T(t+p\beta-\tau)\eta(\tau), \left(\overline
{s} I- A^*\right)\frac{g}{\beta^2}\right\rangle dtd\tau\\
&\quad+\int_{p\beta}^{\infty}  \int_{0}^{\infty}  e^{-s(t+\tau) }\left\langle T(t)\eta(\tau), \left(\overline
{s} I- A^*\right)\frac{g}{\beta^2}\right\rangle dtd\tau\\
=&\int_{0}^{p\beta} e^{-sp\beta} \int_{0}^{\infty}  e^{-st }\left\langle T(t)T(p\beta-\tau)\eta(\tau), \left(\overline
{s}I- A^*\right)\frac{g}{\beta^2}\right\rangle dtd\tau\\
&\quad+\int_{p\beta}^{\infty}  e^{-s\tau }\int_{0}^{\infty}  e^{-st}\left\langle T(t)\eta(\tau), \left(\overline
{s} I- A^*\right)\frac{g}{\beta^2}\right\rangle dtd\tau\\
=&\int_{0}^{p\beta} e^{-sp\beta}  \left\langle T(p\beta-\tau)\eta(\tau),  \frac{g}{\beta^2}\right\rangle d\tau+ \int_{p\beta}^{\infty}  e^{-s\tau } \left\langle \eta(\tau),  \frac{g}{\beta^2}\right\rangle d\tau,
 \end{align*}
 where we changed the order of integration in the first inequality. 

Therefore, for $s> \omega$, 
 \begin{equation}\label{eqn22-background-1}
 \begin{aligned}
 &\int_{\ell\beta}^{(\ell+1)\beta}e^{-s t}\left\langle \int_{0}^{t}T(t-\tau)\eta(\tau)d\tau, \left(\overline
 {s} I- A^*\right)\frac{g}{\beta^2}\right\rangle dt\\
     =&\left(\int_{\ell\beta}^{\infty} -\int_{(\ell+1)\beta}^{\infty}\right)e^{-s t}\left\langle \int_{0}^{t}T(t-\tau)\eta(\tau)d\tau, \left(\overline
     {s}I- A^*\right)\frac{g}{\beta^2}\right\rangle dt\\
     =&\int_{0}^{\ell\beta} e^{-s\ell\beta}  \left\langle T(\ell\beta-\tau)\eta(\tau),  \frac{g}{\beta^2}\right\rangle d\tau+ \int_{\ell\beta}^{\infty}  e^{-s\tau} \left\langle \eta(\tau),  \frac{g}{\beta^2}\right\rangle d\tau-\\
     &\left(\int_{0}^{(\ell+1)\beta} e^{-s(\ell+1)\beta}  \left\langle T((\ell+1)\beta-\tau)\eta(\tau),  \frac{g}{\beta^2}\right\rangle d\tau+ \int_{(\ell+1)\beta}^{\infty}  e^{-s\tau} \left\langle \eta(\tau),  \frac{g}{\beta^2}\right\rangle d\tau\right) \\
     =&\int_{0}^{\ell\beta} e^{-s\ell\beta}  \left\langle T(\ell\beta-\tau)\eta(\tau),  \frac{g}{\beta^2}\right\rangle d\tau+ \int_{\ell\beta}^{(\ell+1)\beta}   e^{-s\tau} \left\langle \eta(\tau),  \frac{g}{\beta^2}\right\rangle d\tau \\
     &-\int_{0}^{(\ell+1)\beta} e^{-s(\ell+1)\beta}  \left\langle T((\ell+1)\beta-\tau)\eta(\tau),  \frac{g}{\beta^2}\right\rangle d\tau
 \end{aligned}
  \end{equation}
  as claimed.
\end{proof}

\begin{proof}[Proof of Lemma \ref{derivations}] As above, let $\omega$ be the growth bound of the semigroup $\{T(t)\}_{t \geq 0}$. Assume, for now, that $s > \omega$ so that the resolvent
\[
R(s,A) = (sI - A)^{-1} = \int_0^{\infty} e^{-s t} T(t) \; dt
\]
is defined at $s$. Using this, we first compute: 
\begin{equation}\label{eq_init}
    \begin{split}
    &\int_{\ell\beta}^{(\ell+1)\beta} e^{-st}\ang{T(t)u_0,(\overline{s}I - A^*)\frac{g}{\beta^2}}dt \\
    =& \int_{0}^\infty e^{-s(t+\ell\beta)}\ang{T(t)T(\ell\beta)u_0,(\overline{s}I - A^*)\frac{g}{\beta^2}}dt  \\
    &\quad-\int_{0}^\infty e^{-s(t+(\ell+1)\beta)}\ang{T(t)T\big((\ell+1)\beta\big)u_0,(\overline{s}I - A^*)\frac{g}{\beta^2}}dt \\
    =& e^{-s\ell\beta}\ang{R(s,A)T(\ell\beta)u_0,(\overline{s}I - A^*)\frac{g}{\beta^2}}\\ 
    &\quad-e^{-s(\ell+1)\beta}\ang{R(s,A)T\big((\ell+1)\beta\big)u_0,(\overline{s}I - A^*)\frac{g}{\beta^2}}\\
    =& e^{-s\ell\beta}\ang{T(\ell\beta)u_0,\frac{g}{\beta^2}} - e^{-s(\ell+1)\beta}\ang{T\big((\ell+1)\beta\big)u_0,\frac{g}{\beta^2}}.
    \end{split}
\end{equation}

Now we need to simplify 
\[
\int_{\ell\beta}^{(\ell + 1)\beta} e^{-st} \chi_{[t_j,\infty)}\int_{t_j}^t e^{\rho_j (t_j - \tau)}\ang{T(t-\tau)h_j,(\overline{s}I-A^*)\frac{g}{\beta^2}}\; d\tau dt.
\]
Define $\alpha = \max{\{\ell\beta,t_j\}}$ and $a(\tau) = \max{\{\alpha-\tau,0\}}$. Then we have the following string of equalities: 
\begin{equation}
    \begin{split}
        &\int_{\ell\beta}^{(\ell + 1)\beta} e^{-st} \chi_{[t_j,\infty)}\int_{t_j}^t e^{\rho (t_j - \tau)}\ang{T(t-\tau)h_j,(\overline{s}I-A^*)\frac{g}{\beta^2}}\; d\tau dt \\
        =&\int_{\alpha}^{(\ell+1)\beta} e^{-st} \int_{t_j}^t e^{\rho (t_j - \tau)}\ang{T(t-\tau)h_j,(\overline{s}I-A^*)\frac{g}{\beta^2}}\; d\tau dt \\
        =& e^{\rho t_j} \int_{t_j}^{(\ell+1)\beta} \int_{\max{\{\alpha,\tau\}}}^{(\ell+1)\beta} e^{-\rho\tau}e^{-st}\ang{T(t-\tau)h_j,(\overline{s}I-A^*)\frac{g}{\beta^2}} \;dtd\tau \\
        =& e^{\rho t_j} \int_{t_j}^{(\ell+1)\beta}e^{-\tau(\rho+s)}\int_{a(\tau)}^{(\ell+1)\beta - \tau}e^{-st}\ang{T(t)h_j,(\overline{s}I-A^*)\frac{g}{\beta^2}}\;dtd\tau \\
        =& e^{\rho t_j} \int_{t_j}^{(\ell+1)\beta} e^{-\tau(\rho+s)}\ang{R(s,A)[e^{-sa(\tau)}T\big(a(\tau)\big)h_j - e^{-s((\ell+1)\beta - \tau}T\big((\ell+1)\beta - \tau\big)h_j],(\overline{s}I - A^*)\frac{g}{\beta^2}} \;d\tau \\
        =& e^{\rho t_j} \int_{t_j}^{(\ell+1)\beta} e^{-\tau(\rho+s)}\ang{e^{-sa(\tau)}T\big(a(\tau)\big)h_j,\frac{g}{\beta^2}} - e^{-\tau \rho}e^{-s(\ell+1)\beta}\ang{T\big((\ell+1)\beta - \tau)\big)h_j,\frac{g}{\beta^2}} \; d\tau. 
    \end{split}
\end{equation}

If $\ell\beta \leq t_j$, then $a(\tau) \equiv 0$ (for $t_j < \tau < (\ell+1)\beta)$. If $t_j < \ell\beta$, then 
\[
a(\tau) = \begin{cases}
    \ell\beta - \tau & \tau \in [t_j,\ell\beta] \\
    0 & \tau \in (\ell\beta,((\ell+1)\beta).
\end{cases}
\]
In the former case, 
\begin{equation}\label{eq_first_with_s}
\begin{split}
& e^{\rho t_j} \int_{t_j}^{(\ell+1)\beta} e^{-\tau(\rho+s)}\ang{e^{-sa(\tau)}T\big(a(\tau)\big)h_j,\frac{g}{\beta^2}} \; d\tau  \\
=& e^{\rho t_j} \int_{t_j}^{(\ell+1)\beta}e^{-\tau(\rho+s)}\ang{h_j,\frac{g}{\beta^2}} \; d\tau \\
=& \frac{e^{-st_j} - e^{\rho(t_j - (\ell+1)\beta) - s(\ell+1)\beta}}{(\rho+s)\beta^2}\ang{h_j,g}.
\end{split}
\end{equation}
In the latter case, 
\begin{equation}\label{eq_second_with_s}
\begin{split}
&e^{\rho t_j} \int_{t_j}^{(\ell+1)\beta} e^{-\tau(\rho+s)}\ang{e^{-sa(\tau)}T\big(a(\tau)\big)h_j,\frac{g}{\beta^2}} \; d\tau \\
=& e^{\rho t_j}\int_{t_j}^{\ell\beta} e^{-\tau(\rho+s)}\ang{e^{-s(\ell\beta - \tau)}T\big(\ell\beta - \tau\big)h_j,\frac{g}{\beta^2}} d\tau\\
&\quad+ e^{\rho}t_j \int_{\ell\beta}^{(\ell+1)\beta} e^{-\tau(\rho+s)}\ang{h_j,\frac{g}{\beta^2}} d\tau\\
=& e^{\rho t_j}\int_{t_j}^{\ell\beta} e^{-\tau\rho}\ang{e^{-s\ell\beta}T\big(\ell\beta - \tau\big)h_j,\frac{g}{\beta^2}}d\tau \\
&\quad+ \frac{e^{\rho(t_j - \ell\beta) - s\ell\beta} - e^{\rho(t_j - (\ell+1)\beta)- s(\ell+1)\beta}}{(\rho+s)\beta^2}\ang{h_j,g}.
\end{split}
\end{equation}

 Putting together (\ref{eq_background}), (\ref{eq_init}), (\ref{eq_first_with_s}) and (\ref{eq_second_with_s}), we have 
\begin{equation}\label{eq_first_total}
\begin{split}
    m_{s,\ell}(g,\beta) &= e^{-s\ell\beta}\ang{T(\ell\beta)u_0,\frac{g}{\beta^2}} - e^{-s(\ell+1)\beta}\ang{T\big((\ell+1)\beta\big)u_0,\frac{g}{\beta^2}} \\
    &+ \frac{e^{-st_j} - e^{\rho(t_j - (\ell+1)\beta) - s(\ell+1)\beta}}{(\rho+s)\beta^2}\ang{h_j,g} + \int_{\ell\beta}^{(\ell+1)\beta} e^{-st}\ang{\eta(\tau),\frac{g}{\beta^2}} dt \\
    &+ \int_0^{\ell\beta}e^{-s\ell\beta}\ang{T(\ell\beta - \tau)\eta(\tau),\frac{g}{\beta^2}}d\tau - \int_0^{(\ell+1)\beta}e^{-s(\ell+1)\beta}\ang{T((\ell+1)\beta - \tau)\eta(\tau),\frac{g}{\beta^2}}d\tau
\end{split}
\end{equation}
when $\ell\beta \leq t_j$, and 
\begin{equation}\label{eq_second_total}
\begin{split}
    m_{s,\ell}(g,\beta) &= e^{-s\ell\beta}\ang{T(\ell\beta)u_0,\frac{g}{\beta^2}} - e^{-s(\ell+1)\beta}\ang{T\big((\ell+1)\beta\big)u_0,\frac{g}{\beta^2}} \\ 
    &+ e^{\rho t_j}\int_{t_j}^{\ell\beta} e^{-\tau\rho}\ang{e^{-s\ell\beta}T\big((\ell+1)\beta - \tau\big)h_j,\frac{g}{\beta^2}} \\
    &+ e^{\rho t_j}\int_{t_j}^{\ell\beta} e^{-\tau\rho}\ang{e^{-s\ell\beta}T\big(\ell\beta - \tau\big)h_j,\frac{g}{\beta^2}}d\tau \\
&+ \frac{e^{\rho(t_j - \ell\beta) - s\ell\beta} - e^{\rho(t_j - (\ell+1)\beta)- s(\ell+1)\beta}}{(\rho+s)\beta^2}\ang{h_j,g} + \int_{\ell\beta}^{(\ell+1)\beta} e^{-st}\ang{\eta(\tau),\frac{g}{\beta^2}} dt \\
    &+ \int_0^{\ell\beta}e^{-s\ell\beta}\ang{T(\ell\beta - \tau)\eta(\tau),\frac{g}{\beta^2}}d\tau - \int_0^{(\ell+1)\beta}e^{-s(\ell+1)\beta}\ang{T((\ell+1)\beta - \tau)\eta(\tau),\frac{g}{\beta^2}}d\tau
\end{split}
\end{equation}
when $t_j < \ell\beta$.

We have thus shown that (\ref{eq_abstract_m}) and (\ref{eq_first_total}) are equivalent when $\ell\beta \leq t_j$, and (\ref{eq_abstract_m}) and (\ref{eq_second_total}) are equivalent when $t_j < \ell\beta$ for all real $s > \omega$. Now, (\ref{eq_abstract_m}), (\ref{eq_first_total}), and (\ref{eq_second_total}) are all analytic functions of $s$. It follows that these equivalences are valid for all $s \in \mathbb{C}$. For $s = \frac{2\pi i k}{\beta}$, (\ref{eq_first_total}) becomes 
\begin{align*}
    m_{s,\ell}(g,\beta) &= \ang{T(\ell\beta)u_0,\frac{g}{\beta^2}} - \ang{T\big((\ell+1)\beta\big)u_0,\frac{g}{\beta^2}} + \frac{e^{-st_j} - e^{\rho(t_j - (\ell+1)\beta) - s(\ell+1)\beta}}{(\rho+s)\beta^2}\ang{h_j,g} \\
    &+ \int_{\ell\beta}^{(\ell+1)\beta} e^{-st}\ang{\eta(\tau),\frac{g}{\beta^2}} dt + \int_0^{\ell\beta}\ang{T(\ell\beta - \tau)\eta(\tau),\frac{g}{\beta^2}} \; d\tau \\
    &- \int_0^{(\ell+1)\beta}\ang{T((\ell+1)\beta - \tau)\eta(\tau),\frac{g}{\beta^2}} \; d\tau
\end{align*}
Similarly, (\ref{eq_second_total}) becomes 
\begin{align*}
    m_{s,\ell}(g,\beta) &= \ang{T(\ell\beta)u_0,\frac{g}{\beta^2}} - \ang{T\big((\ell+1)\beta\big)u_0,\frac{g}{\beta^2}} \\ 
    &+ e^{\rho t_j}\int_{t_j}^{\ell\beta} e^{-\tau\rho}\ang{T\big((\ell+1)\beta - \tau\big)h_j,\frac{g}{\beta^2}} + \frac{e^{\rho(t_j - \ell\beta)} - e^{\rho(t_j - (\ell+1)\beta)}}{(\rho+s)\beta^2}\ang{h_j,g}\\ 
    &+ \int_{\ell\beta}^{(\ell+1)\beta} e^{-st}\ang{\eta(\tau),\frac{g}{\beta^2}} dt + \int_0^{\ell\beta}\ang{T(\ell\beta - \tau)\eta(\tau),\frac{g}{\beta^2}}\; d\tau \\
    &- \int_0^{(\ell+1)\beta}\ang{T((\ell+1)\beta - \tau)\eta(\tau),\frac{g}{\beta^2}}\; d\tau.
\end{align*}
Collecting the terms that don't depend on $s$ proves the theorem. 
\end{proof}

\bigskip 
\noindent {\bf{Acknowledgements}.} The authors dedicate this paper to Charly Gr\"ochening, a great mathematician and friend, 
to commemorate a wonderful event that happened approximately $f(1959)$ years ago, where $$f(t) = \sum\limits^\infty_{m=-\infty}\sum\limits^\infty_{n=0}\frac{m-1+(2^6+1)\mathrm{sinc} (m-1)}{(n+1)\ln 2} 
e^{\frac {i\pi nt}{1959}}e^{-(t-1959m)^2}.$$ 

The authors of the paper were supported in part by the collaborative NSF grants DMS-2208030 and DMS-2208031.

\bibliographystyle{siam}
\bibliography{refs,dynamical-references}

\end{document}